\documentclass[10pt]{amsart}

\usepackage{MnSymbol}

\newcommand{\N}{\mathbb{N}}
\newcommand{\R}{\mathbb{R}}
\newcommand{\C}{\mathbb{C}}
\newcommand{\D}{\mathbb{D}}

\newcommand{\T}{\mathbb{T}}

\newcommand{\U}{\mathbb{U}}

\newtheorem{theorem}{Theorem}
\newtheorem{lemma}[theorem]{Lemma}
\newtheorem{corollary}[theorem]{Corollary}

\theoremstyle{remark}

\renewcommand{\Re}{\operatorname{Re}\,}
\renewcommand{\Im}{\operatorname{Im}\,}

\title[Suffridge's Convolution Theorem for Real Polynomials]{Suffridge's
  Convolution Theorem for Polynomials and Entire Functions Having Only Real
  Zeros} 

\author{Martin Lamprecht} 

\address{Department of Computer Science and Engineering\\
School of Sciences\\
European University of Cyprus\\
Diogenous Str. 6, Engomi, P.O. Box: 22006, 1516 Nicosia, Cyprus}
\email{m.lamprecht@euc.ac.cy}

\thanks{Most of the research that led to this paper was conducted
  while I had a post-doc position at the math department of the University of
  W\"urzburg. I am very grateful to Stephan Ruscheweyh and all members of
  Lehrstuhl IV for the productive and friendly atmosphere there.}



\begin{document}

\maketitle
\begin{abstract}
  We present a Suffridge-like extension of the Grace-Szeg\"o convolution theorem
  for polynomials and entire functions with only real zeros. Our results can
  also be seen as a $q$-extension of P\'olya's and Schur's characterization of
  multiplier sequences. As a limit case we obtain a new characterization of all
  log-concave sequences in terms of the zero location of certain associated
  polynomials. Our results also lead to an extension of Ruscheweyh's convolution
  lemma for functions which are analytic in the unit disk and to new necessary
  conditions for the validity of the Riemann Conjecture.
\end{abstract}

\section{Introduction}
\label{sec:introduction}

In \cite{rota97} Rota states: ``Grace's theorem is an instance of what
might be called a sturdy theorem. For almost one hundred years it has resisted
all attempts at generalization. Almost all known results about the distribution
of zeros of polynomials in the complex plane are corollaries of
Grace's theorem.''

The following equivalent formulation of Grace's theorem is due to Szeg\"o{}.

\begin{theorem}[Grace \mbox{\cite{grace}}, Szeg\"o \mbox{\cite{szegoe}}]
  Let 
  \begin{equation*}
    F(z) = \sum_{k=0}^{n} {n \choose k} a_{k} z^{k}\quad\mbox{and} \quad G(z) =
    \sum_{k=0}^{n} {n \choose k} b_{k} z^{k} 
  \end{equation*}
  be polynomials of degree $n\in\N$ and suppose $K\subset \C$ is an
  open or closed disk or half-plane, or the open or closed exterior of
  a disk, that contains all zeros of $F$. If $G(0)\neq 0$, then each
  zero $\gamma$ of
  \begin{equation*}
    F*_{GS} G(z) := \sum_{k=0}^{n} {n\choose k} a_{k} b_{k} z^{k}
  \end{equation*}
  is of the form $\gamma = -\alpha \beta$ with $\alpha \in K$ and
  $G(\beta)=0$. If $G(0)=0$, then this continues to hold as long as $K$ is
  not the open or closed exterior of a disk.
\end{theorem}

Rota is right (of course): This theorem includes or implies most other known
results concerning the zero location of complex polynomials. It has found
numerous applications in complex analysis and other fields. For instance, it
forms the basis of the geometric convolution theory which was developed by
Ruscheweyh, Suffridge, and Sheil-Small (see
\cite{rusch77,rusch82,ruschsheil73,sheil78,sheil,suffridge} and, more recently,
\cite{ruschsal08,ruschsal09,ruschsalsug09}) and it can be used to classify all
linear operators which preserve the set of polynomials whose zeros lie in a
given circular domain (cf. \cite[Thm. 1.1]{rusch82}, \cite[Sec. 5.8]{sheil}, and
\cite{borcebraend09}). Very recently, in an impressive series of papers
\cite{borcebraend09b,borcebraend09a,borcbraend10}, Borcea and Br\"anden used
Grace's theorem in order to develop a unified analytic theory of multivariate
polynomials with many astonishing applications.

In this paper we will present a real polynomial analogue of a striking extension
of the Grace-Szeg\"o convolution theorem which was found by Suffridge in
\cite{suffridge}. Our result can also be seen as a $q$-extension and a finite
difference analogue \cite{braendKrasShap2012} of P\'{o}lya's and Schur's
\cite{polschur14} famous classification of multiplier sequences. As consequences
we obtain a new classification of all log-concave sequences in terms of the zero
location of certain associated polynomials, several analogues of a convolution
lemma of Ruscheweyh which is of great importance in the convolution theory of
functions analytic in $\D$, and a new continuous connection between the Riemann
Conjecture and a necessary condition of it which was verified by Csordas,
Norfolk, and Varga in \cite{csnorvar1986}.

We believe that Suffridge's work \cite{suffridge}, the recent work of
Ruscheweyh and Salinas \cite{ruschsal08,ruschsal09,ruschsalsug09}, and the
results of this paper and \cite{lam11} (the methods of proof presented here
and in \cite{lam11} also seem to have some kind of resemblance to the methods
used in \cite{garwag96}), strongly hint at a very deep lying extension of
Grace's theorem which will lead to a much better understanding of the
relation between the zeros and the coefficients of complex polynomials.

\subsection{Special cases of the Grace-Szeg\"o convolution theorem}
\label{sec:special-cases-grace}

As usual, for a field $\mathbb{K}$ we denote the set of polynomials of degree
$\leq n$ (this includes the polynomial identically $0$ which is of degree $-1$)
over $\mathbb{K}$ by $\mathbb{K}_{n}[z]$ (the only fields $\mathbb{K}$ that will
be considered in this paper are $\C$ and $\R$). $\mathbb{K}[z]$ denotes the set
of all polynomials over $\mathbb{K}$ and $\mathbb{K}[[z]]$ is the set of formal
power series over $\mathbb{K}$. If $f\in \mathbb{K}[[z]]\setminus
\mathbb{K}[z]$, then we set $\deg f:= +\infty$. The \emph{convolution} or
\emph{Hadamard product} of $f(z) = \sum_{k=0}^{\infty} a_{k} z^{k}$, $g(z) =
\sum_{k=0}^{\infty} b_{k} z^{k}\in\C[[z]]$ is defined by
\begin{equation*}
  f*g (z) = \sum_{k=0}^{\infty} a_{k} b_{k} z^{k}.
\end{equation*}
The \emph{multiplier class} $\mathcal{M}(\mathcal{X})$ of a subset $\mathcal{X}$
of $\C[[z]]$ consists of those $g\in\C[[z]]$ with $\deg g \leq \max\{\deg h:
h\in\mathcal{X}\}$ which have the property that $f*g\in\mathcal{X}$ for all
$f\in\mathcal{X}$. 

For an unbounded subset $\Omega$ of $\C$ we define $\pi_{n}(\Omega)$ to be the
set of all polynomials in $\C_{n}[z]$ which have zeros only in $\Omega$. If
$\Omega$ is bounded, then $\pi_{n}(\Omega)$ shall contain all polynomials of
degree $n$ with zeros only in $\Omega$. For every $\Omega\subset \C$ the class
$\pi_{n}(\Omega)$ shall also contain the polynomial identically
zero. $\sigma_{n}(\Omega)$ will denote the union of $\{0\}$ with the set of all
polynomials in $\pi_{n}(\Omega)$ which have only simple zeros and which, in the
case that $\Omega$ is unbounded, are of degree $n$ or $n-1$. Finally, for
$\mathcal{X}\subset\C[[z]]$ and $h\in \C[[z]]$ we denote by
$\mathcal{P}(\mathcal{X};h)$ the \emph{pre-coefficient class} of $\mathcal{X}$
with respect to $h$, i.e.  those $f\in\C[[z]]$ with $\deg f \leq \deg h$ for
which $f*h\in \mathcal{X}$.

Several interesting special cases of the Grace-Szeg\"o convolution theorem can
now be stated as follows (for a detailed proof see \cite[Ch. 5]{rahman}). We use
the notations $\R^{\pm}:=\{z\in\R: \pm z >0\}$, $\R_{0}^{\pm}:= \R^{\pm}
\cup\{0\}$, $\D:=\{z\in\C: |z|<1\}$, $\T:=\{z\in\C: |z|=1\}$, and
$\hat{\pi}_{n}(\Omega):=\mathcal{P}(\pi_{n}(\Omega);(1+z)^{n})$ for $\Omega
\subset \C$.

\begin{corollary}
  \label{sec:introduction-1}
  \begin{enumerate}
  \item[]
  \item \label{item:1} $\mathcal{M}(\pi_{n}(\D)) =
    \hat{\pi}_{n}(\overline{\D})$.
  \item \label{item:3} $\mathcal{M}(\pi_{n}(\T)) =
    \hat{\pi}_{n}(\T)$.
  \item \label{item:6} If $H$ is an open half-plane with $0\in \partial H$, then
    $\mathcal{M}(\pi_{n}(H)) = \hat{\pi}_{n}(\R^{-})$.
  \item \label{item:7} $\mathcal{M}(\pi_{n}(\R)) =
    \hat{\pi}_{n}(\R_{0}^{-})\cup \hat{\pi}_{n}(\R_{0}^{+})$.
  \item \label{item:8} $\mathcal{M}(\pi_{n}(\R^{-})) =
    \hat{\pi}_{n}(\R^{-})$.
  \end{enumerate}
\end{corollary}

\subsection{Suffridge's extension of the unit circle case}
\label{sec:suffr-extens-unit}

Evidently, the binomial coefficients and their generating polynomial 
\begin{equation*}
  (1+z)^{n} = \sum_{k=0}^{n} {n\choose k} z^{k}
\end{equation*}
play an essential role in the Grace-Szeg\"o convolution theorem. A particularly
interesting extension of the binomial coefficients are the \emph{$q$-binomial}
or \emph{Gaussian binomial coefficients} $C_{k}^{n}(q)$ which are defined by
\begin{equation}
   \label{eq:23}
   R_{n}(q;z) := \sum_{k=0}^{n} C_{k}^{n}(q) z^{k} 
   := \prod_{j=1}^{n}(1+q^{j-1}z), \qquad q\in\C,
\end{equation}
and take the explicit form \cite[(10.0.5)]{andaskroy99}
\begin{equation}
  \label{eq:33}
  C_{k}^{n}(q) = q^{k(k-1)/2} \prod_{j=1}^{k}\frac{1-q^{j+n-k}}
  {1-q^{j}}, \quad k\in\{0,\ldots,n\}.
\end{equation}
Observe that often (for instance in \cite{andaskroy99})
$q^{-k(k-1)/2}C_{k}^{n}(q)$ are defined to be the $q$-binomial coefficients. If
$q\in\T$, then all zeros of $R_{n}(q;z)$ lie on the unit circle and are
separated by a certain angle. In \cite{suffridge} Suffridge considered
subclasses of $\pi_{n}(\T)$ in which $R_{n}(e^{i\lambda};z)$, with
$\lambda\in[0,\frac{2\pi}{n}]$, is an extremal element.

In order to be more exact, for $n\in\N$ and $\lambda\in[0,\frac{2\pi}{n}]$ we
define the classes $\mathcal{T}_{n}(\lambda)$ to consist of all polynomials
$F\in\pi_{n}(\T)$ which have the property that if $z_{1}$, $z_{2}\in\T$ are
zeros of $F$ (the zeros, as always in this paper, counted according to
multiplicity), then $z_{1}$ and $z_{2}$ are separated by an angle $>\lambda$. We
also define $0$ to be an element of $\mathcal{T}_{n}(\lambda)$. The closure
$\overline{\mathcal{T}}_{n}(\lambda)$ of $\mathcal{T}_{n}(\lambda)$ then
contains $0$ and all polynomials in $\pi_{n}(\T)$ whose zeros are separated by
an angle $\geq \lambda$. The classes $\overline{\mathcal{T}}_{n}(\lambda)$ were
introduced by Suffridge in \cite{suffridge} (where they were denoted by
$\mathcal{P}_{n}(\lambda)$, however, and did not contain $0$).

Every pair (except one) of successive zeros of
\begin{equation}
   \label{eq:31}
   Q_{n}(\lambda;z) :=
   \prod_{j=1}^{n}(1+e^{i(2j-n-1)\lambda/2}z) = R_{n}(e^{i\lambda};e^{-i(n-1)\lambda/2}z)
\end{equation}
is separated by an angle of exactly $\lambda$. This is the reason why, as
indicated above, we call a polynomial $F$ in
$\overline{\mathcal{T}}_{n}(\lambda)$ \emph{extremal} if there is an $a\in\T$
such that $F(z) =_{\C} Q_{n}(\lambda;az)$, where, from now on, for $F$,
$G\in\C[[z]]$ and $\mathbb{K} = \R $ or $\mathbb{K} = \C$ we write $F
=_{\mathbb{K}} G$ if there is an $a\in \mathbb{K}\setminus\{0\}$ such that
$F=aG$. For $\lambda\in[0,\frac{2\pi}{n})$ we set $\mathcal{PT}_{n}(\lambda):=
\mathcal{P}(\mathcal{T}_{n}(\lambda);Q_{n}(\lambda;z))$, while
\begin{equation*}
  \mathcal{PT}_{n}\left(\frac{2\pi}{n}\right) :=
  \bigcup_{\lambda\in[0,\frac{2\pi}{n})} \mathcal{PT}_{n}(\lambda).
\end{equation*}
We call a polynomial $f\in \overline{\mathcal{PT}}_{n}(\lambda)$
\emph{extremal} if $f*Q_{n}(\lambda;z)$ is extremal in
$\overline{\mathcal{T}}_{n}(\lambda)$, i.e. if there is an $a\in\T$
such that $f(z) =_{\C} e_{n}(az)$ with
\begin{equation*}
  e_{n}(z) = 1 + z+ \cdots +z^{n-1} + z^{n}.
\end{equation*}

Suffridge's stunning results from \cite{suffridge} now read as follows.

\begin{theorem}[Suffridge]
  \label{sec:suffr-main-thm}
  Let $\lambda\in[0,\frac{2\pi}{n}]$.
  \begin{enumerate}
  \item \label{item:17} We have $\mathcal{M}(\mathcal{PT}_{n}(\lambda))
    =\overline{\mathcal{PT}}_{n}(\lambda)$. In particular, for
    $\lambda\in[0,\frac{2\pi}{n})$ we have
    $\mathcal{M}(\mathcal{T}_{n}(\lambda))
    =\overline{\mathcal{PT}}_{n}(\lambda)$.
  \item \label{item:18} If $\mu\in(\lambda,\frac{2\pi}{n}]$ and $f\in
    \overline{\mathcal{PT}}_{n}(\lambda)$ is not extremal, then
    $f\in\mathcal{PT}_{n}(\mu)$.
  \item \label{item:19} We have
    \begin{equation*}
      \overline{\mathcal{PT}}_{n}\left(\frac{2\pi}{n}\right) =
      \bigcup_{a\in \T,b\in\C} 
      \operatorname{co}\{b\, e_{n}(e^{2ij\pi/n}az):j=1,\ldots,n\}, 
    \end{equation*}
    where $\operatorname{co} M$ denotes the convex hull of a subset
    $M$ of a complex vector space.
  \end{enumerate}
\end{theorem}

Since $\overline{\mathcal{T}}_{n}(0) = \pi_{n}(\T)$, $\mathcal{T}_{n}(0) =
\sigma_{n}(\T)\subset\pi_{n}(\T)$, and $Q_{n}(0;z) = (1+z)^{n}$,
(\ref{eq:23}) and (\ref{eq:31}) show that Theorem
\ref{sec:suffr-main-thm}(\ref{item:17}) can be seen as a $q$-extension of
Corollary \ref{sec:introduction-1}(\ref{item:3}).

Naturally, this extension of Corollary \ref{sec:introduction-1}(\ref{item:3})
triggers the question whether there are other statements of Corollary
\ref{sec:introduction-1} that can be generalized in a similar way. In this paper
we will show how to obtain $q$-extensions (for real $q$) of Statements
(\ref{item:6})--(\ref{item:8}) of Corollary
\ref{sec:introduction-1}(\ref{item:3}) by modifying the proof of Suffridge's
theorem that is given in \cite{lam11}.

\section{Main results}
\label{sec:main-results}

\subsection{Suffridge's theorem for real polynomials}
\label{sec:suffr-theor-real}

The main idea for obtaining a real polynomial version of Suffridge's theorem is
to consider $R_{n}(q;z)$ with $q\in[0,1]$ as an extremal polynomial for certain
classes of real polynomials.

Recall that
\begin{equation*}
  R_{n}(q;z) =\sum_{k=0}^{n} C_{k}^{n}(q) z^{k} =
  \prod_{j=1}^{n}(1+q^{j-1}z), \qquad n\in\N,\; q\in[0,1].
\end{equation*}
Hence, for $q\in(0,1]$ the zeros $x_{j}:=-q^{-j}$, $j\in\{0,\ldots,n-1\}$, of
$R_{n}(q;z)$ satisfy the separation condition $x_{j}/x_{k}\leq q$ for
$k>j$. If we suppose $R_{n}(q;z)$ to be extremal for a certain class of
real polynomials, we are therefore led to the following definitions.

For $q\in[0,1]$ we call a finite or infinite sequence $\{x_{k}\}_{k}$ of real
numbers \emph{logarithmically $q$-separated}, or shorter \emph{$q$-separated},
if $x_{k}/x_{l} \leq q$ for all indices $k$, $l$ with $k\neq l$ for which either
$x_{l}\leq x_{k}<0$ or $0<x_{k}\leq x_{l}$ holds. If $x_{k}/x_{l} < q$ for all
such indices $k$, $l$, then $\{x_{k}\}_{k}$ is called \emph{strictly
  logarithmically $q$-separated}, or \emph{strictly $q$-separated}. For $n\in\N$
and $q\in[0,1]$ we define $\mathcal{R}_{n}(q)$ as the union of $\{0\}$ with the
set of \emph{real} polynomials in $\pi_{n}(\R)$ that have strictly
$q$-separated zeros. $\mathcal{N}_{n}(q)$ will denote $\mathcal{R}_{n}(q)\cap
\pi_{n}(\R^{-}_{0})$. $R_{n}(q;z)$ belongs to both
$\overline{\mathcal{R}}_{n}(q)$ and $\overline{\mathcal{N}}_{n}(q)$, and we call
a polynomial $F$ in one of these two classes extremal if there is an $a\in
\R\setminus\{0\}$ such that $F(z) =_{\R} R_{n}(q;az)$. For $q\in(0,1]$ we
further set $\mathcal{PR}_{n}(q):=\mathcal{P}(\mathcal{R}_{n}(q);R_{n}(q;z))$,
$\mathcal{PN}_{n}(q):=\mathcal{P}(\mathcal{N}_{n}(q);R_{n}(q;z))$,
\begin{equation*}
  \mathcal{PR}_{n}(0) := \bigcup_{q\in(0,1]} \mathcal{PR}_{n}(q) 
  \quad\mbox{and}\quad 
  \mathcal{PN}_{n}(0) := \bigcup_{q\in(0,1]} \mathcal{PN}_{n}(q).
\end{equation*}
For $n\in\N\cup\{\infty\}$ we also define $\mathcal{LC}_{n}$ to consist of
those $\sum_{k=0}^{n} a_{k} z^{k}\in\R[[z]]$ which satisfy
\begin{equation*}
  a_{k}^{2} > a_{k-1} a_{k+1} 
\end{equation*}
for all $0\leq k< n+1$ for which there are $l\leq k$ and $m\geq k$ with $a_{l}$,
$a_{m}\neq 0$. $\mathcal{LC}_{n}^{+}$ shall be the set of those $\sum_{k=0}^{n}
a_{k} z^{k}\in\mathcal{LC}_{n}$ for which $a_{k}\geq 0$ for all $k$ or
$a_{k}\leq 0$ for all $k$. Then $\overline{\mathcal{LC}_{n}}$ contains all
formal power series (or polynomials) whose coefficient sequences
$\{a_{k}\}_{k=0}^{n}$ satisfy $a_{k}^{2}\geq a_{k-1}a_{k+1}$ for all $0\leq k <
n+1$. Such sequences are usually called \emph{log-concave} and
$\mathcal{LC}_{n}$ contains the \emph{strictly log-concave} sequences. Observe
that $0\in \mathcal{LC}_{n}^{+}\subset\mathcal{LC}_{n}\subset
\overline{\mathcal{LC}}_{n}$ and that every $f\in\overline{LC^{+}}_{\infty}$ has
positive radius of convergence (\cite[Ch. 8 Thm. 1.1]{kar68})
 
The above definitions imply that, for instance, 
\begin{equation}
  \label{eq:28}
  \overline{\mathcal{R}}_{n}(1) = \pi_{n}(\R) \quad\mbox{and}\quad 
\overline{\mathcal{N}}_{n}(1)=\pi_{n}(\R^{-}_{0});
\end{equation}
furthermore, $\overline{\mathcal{R}}_{n}(0) = \mathcal{R}_{n}(0) =
\overline{\mathcal{N}}_{n}(0) = \mathcal{N}_{n}(0) =
\{0,1,z,z^{2},\ldots,z^{n}\}$, and if $F\in\mathcal{R}_{n}(q)$ for a
$q\in[0,1]$, then all zeros of $F$ are simple except possibly a
multiple zero at the origin.

The main result of this paper is the following analogue of Theorem
\ref{sec:suffr-main-thm} for the classes $\overline{\mathcal{R}}_{n}(q)$ and
$\overline{\mathcal{N}}_{n}(q)$. Because of (\ref{eq:28}), Statements
(\ref{item:26}) and (\ref{item:27}) of the theorem below are the desired
$q$-extensions of Corollary \ref{sec:introduction-1}(\ref{item:7}) and
(\ref{item:8}).

\begin{theorem}
  \label{sec:real-zeros-main-thm}
  Let $q\in[0,1]$ and $n\in\N$.
  \begin{enumerate}
  \item \label{item:26} We have $\mathcal{M}(\mathcal{PR}_{n}(q))= \{ f(\pm z):
    f\in\overline{\mathcal{PN}}_{n}(q)\}$. In particular, if $q\in (0,1]$, then
    $\mathcal{M}(\mathcal{R}_{n}(q))=\{ f(\pm z):
    f\in\overline{\mathcal{PN}}_{n}(q)\}$.
  \item \label{item:27} We have $\mathcal{M}(\mathcal{PN}_{n}(q))=
    \overline{\mathcal{PN}}_{n}(q)$. In particular, if $q\in (0,1]$, then
    $\mathcal{M}(\mathcal{N}_{n}(q))= \overline{\mathcal{PN}}_{n}(q)$.
  \item \label{item:28} If $r\in[0,q)$ and if $f$ is not extremal and belongs to
    $\overline{\mathcal{PR}}_{n}(q)$ or $\overline{\mathcal{PN}}_{n}(q)$, then
    $f$ is also an element of $\mathcal{PR}_{n}(r)$ or $\mathcal{PN}_{n}(r)$,
    respectively.
  \item \label{item:29} We have
    \begin{equation*}
      \mathcal{PR}_{n}(0) = \mathcal{LC}_{n} \quad \mbox{and} \quad
      \mathcal{PN}_{n}(0) = \mathcal{LC}_{n}^{+}.
    \end{equation*}
  \end{enumerate}
\end{theorem}

Statements (\ref{item:26})--(\ref{item:28}) of this theorem are obtained as
corollaries of certain interspersion invariance results concerning the classes
$\overline{\mathcal{R}}_{n}(q)$ and $\overline{\mathcal{N}}_{n}(q)$ (Theorems
\ref{sec:thm-inv-lhd-vee} and \ref{sec:thm-inv-prec-prec}). Together with the
Hermite-Biehler theorem (cf. \cite[Thm. 6.3.4]{rahman}), these results also lead
to a $q$-extension of Corollary \ref{sec:introduction-1}(\ref{item:6}) (Theorem
\ref{sec:q-halfplane-ext}). Details will be given in Section
\ref{sec:weight-hadam-prod}.

\subsection{A completion of P\'{o}lya's and Schur's characterization of
  multiplier sequences}
\label{sec:completion-}

Letting $n\rightarrow\infty$ in Theorem \ref{sec:real-zeros-main-thm} leads to
the classification of multiplier classes for certain subclasses of real entire
functions of order $0$. For, if $q\in(0,1)$ and $\{x_{j}\}_{j\in \N}$ is a
logarithmically $q$-separated sequence of real numbers for which
$a:=\inf_{j\in\N} |x_{j}|>0$, then
\begin{equation*}
  \sum_{j=1}^{\infty}\frac{1}{|x_{j}|^{\lambda}} \leq \frac{2}{a^{\lambda}} 
  \sum_{j=1}^{\infty}q^{j\lambda}< \infty \quad \mbox{for all} 
  \quad \lambda>0.
\end{equation*}
Consequently, if $n\in\N\cup\{\infty\}$, $a\in\R$, $m\in\N$, and if
$\{x_{j}\}_{j=1}^{n}$ is logarithmically $q$-separated with $\inf_{1\leq j < n+1}
|x_{j}|>0$, then
\begin{equation}
  \label{eq:30}
  F(z) = a z^{m}\prod_{j=1}^{n} \left(1-\frac{z}{x_{j}}\right)
\end{equation}
is an entire function of order $0$. We will denote the set of these entire
functions by $\overline{\mathcal{R}}_{\infty}(q)$, and define
$\overline{\mathcal{N}}_{\infty}(q)$ to be the set of those functions in
$\overline{\mathcal{R}}_{\infty}(q)$ which have only non-positive zeros. It is
clear that, for $q\in(0,1)$,
\begin{equation*}
  \overline{\mathcal{R}}_{\infty}(q) = \overline{\bigcup_{n\in\N} 
    \mathcal{R}_{n}(q)} \quad \mbox{and} \quad 
  \overline{\mathcal{N}}_{\infty}(q) = \overline{\bigcup_{n\in\N} 
    \mathcal{N}_{n}(q)} 
\end{equation*}
in the topology of compact convergence in $\C$. On the other hand, if
\begin{equation*}
  \overline{\mathcal{R}}_{\infty}(1) := \overline{\bigcup_{n\in\N} 
    \mathcal{R}_{n}(1)} \quad \mbox{and} \quad 
  \overline{\mathcal{N}}_{\infty}(1) := \overline{\bigcup_{n\in\N} 
    \mathcal{N}_{n}(1)},
\end{equation*}
then it is easy to see that every
$f\in\overline{\mathcal{R}}_{\infty}(1)$ can be approximated,
uniformly on compact subsets of $\C$, by a sequence of polynomials
$F_{n}\in
\overline{\mathcal{R}}_{n}(q_{n})\subset\overline{\mathcal{R}}_{\infty}(q_{n})$
with $q_{n}\rightarrow 1$ as $n\rightarrow\infty$. This implies
\begin{equation*}
  \overline{\mathcal{R}}_{\infty}(1) = \overline{\bigcup_{q\in(0,1)} 
    \mathcal{R}_{\infty}(q)} \quad \mbox{and} \quad 
  \overline{\mathcal{N}}_{\infty}(1) = \overline{\bigcup_{q\in(0,1)} 
    \mathcal{N}_{\infty}(q)}. 
\end{equation*}

The entire function
\begin{equation*}
  R_{\infty}(q;z) := \sum_{k=0}^{\infty} C_{k}^{\infty} (q) z^{k}:= 
  \prod_{j=1}^{\infty} (1+q^{j-1}z) = \lim_{n\rightarrow\infty} R_{n}(q;z), 
  \quad q\in(0,1),
\end{equation*}
belongs to both $\overline{\mathcal{R}}_{\infty}(q)$ and
$\overline{\mathcal{N}}_{\infty}(q)$. It follows from (\ref{eq:33}) that 
\begin{equation*}
  C_{k}^{\infty}(q) = \lim_{n\rightarrow\infty} C_{k}^{n}(q) = 
  \lim_{n\rightarrow\infty} \prod_{j=1}^{k} q^{j-1} \frac{1-q^{j+n-k}}{1-q^{j}}=
  \prod_{j=1}^{k} \frac{q^{j-1}}{1-q^{j}}, \qquad q\in(0,1), \; k\in\N.
\end{equation*}
Consequently,
\begin{equation*}
  (1-q)^{k}C_{k}^{\infty}(q) = \prod_{j=1}^{k} \frac{q^{j-1}(1-q)}{1-q^{j}} 
  \rightarrow \frac{1}{k!} \quad\mbox{as}\quad q\rightarrow 1,
\end{equation*}
and thus, uniformly on compact subsets of $\C$,
\begin{equation*}
  R_{\infty}(q;(1-q)z) \rightarrow e^{z} =: R_{\infty}(1;z)
  \quad\mbox{as}\quad q\rightarrow 1.
\end{equation*}
Hence, if we set
\begin{equation*}
  \overline{\mathcal{PR}}_{\infty}(q):=
  \mathcal{P}(\overline{\mathcal{R}}_{\infty}(q);R_{\infty}(q;z)) \quad \mbox{and}
  \quad \overline{\mathcal{PN}}_{\infty}(q):=
  \mathcal{P}(\overline{\mathcal{N}}_{\infty}(q);R_{\infty}(q;z))
\end{equation*}
for $q\in(0,1]$, we obtain the following from Theorem
\ref{sec:real-zeros-main-thm}.

\begin{theorem}
  \label{sec:real-zeros-inf-thm}
  Let $q\in(0,1]$.
  \begin{enumerate}
  \item \label{item:35} We have
    $\mathcal{M}(\overline{\mathcal{R}}_{\infty}(q))= \{ f(\pm z):
    f\in\overline{\mathcal{PN}}_{\infty}(q)\}$.
  \item \label{item:36} We have
    $\mathcal{M}(\overline{\mathcal{N}}_{\infty}(q))=
    \overline{\mathcal{PN}}_{\infty}(q)$.
  \item \label{item:37} If $r\in(0,q)$ and $f$ belongs to
    $\overline{\mathcal{PR}}_{\infty}(q)$ or
    $\overline{\mathcal{PN}}_{\infty}(q)$, then, respectively, $f$
    belongs to $\overline{\mathcal{PR}}_{\infty}(r)$ or
    $\overline{\mathcal{PN}}_{\infty}(r)$.
  \item \label{item:38} We have
    \begin{equation*}
      \bigcup_{q\in(0,1]}\overline{\mathcal{PR}}_{\infty}(q)  \subset 
      \overline{\mathcal{LC}}_{\infty}
      \quad \mbox{and}  \quad 
      \bigcup_{q\in(0,1]} \overline{\mathcal{PN}}_{\infty}(q) \subset
      \overline{\mathcal{LC}^{+}}_{\infty}.
    \end{equation*}
  \end{enumerate}
\end{theorem}

The cases $q=1$ of Theorem \ref{sec:real-zeros-main-thm}(\ref{item:26}) and
Theorem \ref{sec:real-zeros-inf-thm}(\ref{item:35}) were first obtained by
P\'{o}lya and Schur in \cite{polschur14} and they called them, respectively, the
\emph{algebraic characterization of multiplier sequences of the first kind} and
the \emph{transcendental characterization} of these sequences. Theorems
\ref{sec:real-zeros-main-thm}(\ref{item:26}) and
\ref{sec:real-zeros-inf-thm}(\ref{item:35}) thus represent a $q$-extension and a
finite difference analogue (cf. \cite{braendKrasShap2012}) of P\'{o}lya's and
Schur's characterization of multiplier sequences.

Note also that, as a limit case of Theorem \ref{sec:suffr-main-thm}, in
\cite{suffridge}, Suffridge obtained a second proof (the first one was given by
by Ruscheweyh and Sheil-Small \cite{ruschsheil73}, see also
\cite{lewis,rusch82,sheil}) of a conjecture of P\'{o}lya and Schoenberg from
\cite{polschoe58} which claimed that the convolution of two convex univalent
function is again convex univalent. Statements (\ref{item:35}) and
(\ref{item:36}) of Theorem \ref{sec:real-zeros-inf-thm} can thus also be seen as
the real entire function analogues of the P\'{o}lya-Schoenberg conjecture.

\subsection{A new characterization of log-concave sequences}
\label{sec:new-char-log}

Log-concave sequences play an important role in combinatorics, algebra,
geometry, computer science, probability, and statistics (see
\cite{brenti94,stanley89}), and therefore Theorems
\ref{sec:real-zeros-main-thm}(\ref{item:29}) and
\ref{sec:real-zeros-inf-thm}(\ref{item:38}) might have far-reaching
applications. An important tool for establishing the log-concavity of a
given sequence $\{a_{k}\}_{k}$ of real numbers are ''Newton's inequalities''
(see \cite[Thm. 2]{stanley89}), which state that $\{a_{k}\}_{k=0}^{n}$,
$n\in\N$, is log-concave, if $\sum_{k=0}^{n} {n\choose k} a_{k} z^{k}$ is
a real polynomial with only real zeros. This sufficient condition for
log-concavity is however far from necessary.

As a corollary to Theorems \ref{sec:real-zeros-main-thm}(\ref{item:29}) and
\ref{sec:real-zeros-inf-thm}(\ref{item:38}) we obtain the following new
characterization of \emph{all} log-concave sequences in terms of the zero
location of certain associated polynomials.

\begin{corollary}
  \label{sec:charct-log-conc-seq}
  Let $n\in\N$ and suppose $\{a_{k}\}_{k=0}^{n}$ is a sequence of real
  numbers. Then $\{a_{k}\}_{k=0}^{n}$ is strictly log-concave if, and
  only if, there is a $q\in(0,1]$ such that
  \begin{equation*}
    \sum_{k=0}^{n} C_{k}^{n}(q) a_{k} z^{k}\quad\mbox{belongs to}\quad 
    \mathcal{R}_{n}(q). 
  \end{equation*}
  If all $a_{k}$ are non-negative, then $\sum_{k=0}^{n} C_{k}^{n}(q) a_{k}
  z^{k}$ belongs to $\mathcal{N}_{n}(q)$.

  Moreover, if there is a $q\in(0,1]$ and an infinite sequence
  $\{a_{k}\}_{k=0}^{\infty}$ of real numbers such that
  \begin{equation*}
    \sum_{k=0}^{\infty} C_{k}^{\infty}(q) a_{k} z^{k}\quad\mbox{belongs to}\quad 
    \overline{\mathcal{R}}_{\infty}(q), 
  \end{equation*}
  then $\{a_{k}\}_{k=0}^{\infty}$ is strictly log-concave.  
\end{corollary}

Using the Hermite-Biehler theorem \cite[Thm. 6.3.4]{rahman} and Lemma
\ref{sec:working-2}, one sees that in order to verify whether a given sequence
$\{a_{k}\}_{k=0}^{n}$ is strictly log-concave it is also sufficient to check
whether all zeros $\neq 0$ of the polynomial
\begin{equation*}
  \sum_{k=0}^{n} \left(C_{k}^{n}(q) a_{k} + 
    i q^{k-1} C_{k-1}^{n}(q) a_{k-1}\right) z^{k}
\end{equation*}
lie in the open upper half-plane $\U$ or in the open lower half-plane
$\mathbb{L}$ (if $a_{k}\geq 0$ for all $k\in\{0,\ldots,n\}$ it is even enough to
check whether all zeros $\neq 0$ of the polynomial $ \sum_{k=0}^{n} C_{k}^{n}(q)
a_{k} \left(1+iq^{k}\right) z^{k}$ lie in $\U$ or $\mathbb{L}$).

\subsection{An extension of Ruscheweh's convolution lemma}
\label{sec:an-extens-rusch-1}

The following lemma of Ruscheweyh from \cite{rusch72} plays a fundamental role
in the convolution theory for functions which are analytic in $\D$ (see
\cite{rusch82}). $\mathcal{H}(\D)$ denotes the set of functions analytic in $\D$
and $\mathcal{H}_{0}(\D)$ is the set of those functions $f\in\mathcal{H}(\D)$
which satisfy $f(0)=f'(0)-1=0$.

\begin{lemma}[Ruscheweyh]
  \label{sec:rusch-lemma}
  Suppose $f$, $g\in\mathcal{H}_{0}(\D)$ satisfy
  \begin{equation*}
    \left(f*\frac{1+xz}{1+yz} g\right)(z)\neq 0 \quad \mbox{for all}\quad 
    z\in \D,\; x,y\in\T.
  \end{equation*}
  Then for every $g\in \mathcal{H}(\D)$ we have
  \begin{equation*}
    \frac{f*h}{f*g}(\D) \subset \overline{\operatorname{co}} 
    \left(\frac{h}{g}(\D)\right).
  \end{equation*}
\end{lemma}

Analogues of this lemma for real polynomials (Lemmas \ref{sec:main-results-2}
and \ref{sec:polynomials-with-log-3}) will play a crucial role in our proof of
Theorem \ref{sec:real-zeros-main-thm}. We will prove these analogues in Sections
\ref{sec:polyn-with-intersp} and \ref{sec:polynomials-with-log}. In Section
\ref{sec:proofs-theorems} we will explain how Lemma \ref{sec:main-results-2} can
be used to obtain the following generalization of Ruscheweyh's lemma.

\begin{lemma}
  \label{sec:ext-rusch-lemma}
  Let $L:\mathcal{H}(\D)\rightarrow\mathcal{H}(\D)$ be a continuous
  complex linear operator. Suppose $f\in\mathcal{H}(\D)$ is such that
  \begin{equation}
    \label{eq:62}
    L\left[\frac{1 + x z}{1 + y z}f\right](z) \neq 0 \quad \mbox{for
      all} \quad z\in\D,\, x,\,y\in\T,
  \end{equation}
  and 
  \begin{equation}
    \label{eq:63}
    \left|\frac{L[\frac{zf}{1+yz}]}{L[\frac{f}{1+yz}]}\right|(0)<1
    \quad \mbox{for at least one} \quad y\in\T.
  \end{equation}
  Then for every $g\in \mathcal{H}(\D)$ we have
  \begin{equation*}
    \frac{L[h]}{L[g]}(\D) \subset \overline{\operatorname{co}} 
    \left(\frac{h}{g}(\D)\right),
    \quad z\in\D.
  \end{equation*}
\end{lemma}

\subsection{Consequences regarding the Riemann Conjecture}
\label{sec:cons-riem-conj}

It is well known (and explained in \cite{csnorvar1986}, for example) that
the Riemann Conjecture is equivalent to the statement that
\begin{equation*}
  F(z)= \sum_{n=0}^{\infty} \frac{\hat{b}_{n} z^{n}}{(2n)!}
\end{equation*}
belongs to $\overline{\mathcal{N}}_{\infty}(1)$, where
\begin{equation*}
  \hat{b}_{n}:= \int_{0}^{\infty} t^{2n} \Phi(t)\, dt, \quad n\in\N_{0}, 
  \quad\mbox{and}\quad
  \Phi(t):= \sum_{n=1}^{\infty} (2n^{4}\pi^{2}e^{9t}-3n^{2}\pi e^{5t}) 
  e^{-n^{2}\pi e^{4t}}.
\end{equation*}

A particular consequence of Theorem \ref{sec:real-zeros-inf-thm}
concerning the Riemann Conjecture is the following.

\begin{theorem}
  If the Riemann Conjecture is true, then
  \begin{equation*}
    f(z)= \sum_{n=0}^{\infty}  \frac{n!\hat{b}_{n} z^{n}}{(2n)!}
    \in \overline{\mathcal{PN}}_{\infty}(q) \subset  
    \overline{\mathcal{LC}^{+}}_{\infty}\quad \mbox{for all} \quad q\in(0,1].
  \end{equation*}
\end{theorem}

The statement $f\in \overline{\mathcal{PN}}_{\infty}(q)$ is a necessary
condition for the validity of the Riemann Conjecture, that becomes weaker as $q$
decreases from $1$ to $0$. Its weakest form ($f\in
\overline{\mathcal{LC}^{+}}_{\infty}$) is true due to Csordas, Norfolk, and
Varga \cite{csnorvar1986}.

\subsection{Structure of the paper}
\label{sec:structure-paper}

In the next section we introduce some terminology and notation regarding zeros
and poles of polynomials and rational functions. In Section
\ref{sec:polyn-with-intersp} we establish certain facts regarding polynomials
with interspersed zeros, and obtain, as main result, Lemma
\ref{sec:main-results-2} (the real polynomial version of Ruscheweyh's
convolution lemma). In Section \ref{sec:polynomials-with-log} we prove certain
analogues of results from Section \ref{sec:polyn-with-intersp} for polynomials
with log-interspersed zeros. The main result, Lemma
\ref{sec:polynomials-with-log-3}, is also an analogue of Ruscheweyh's
convolution lemma. In Section \ref{sec:line-oper-pres-1} several auxiliary
results concerning the classes $\overline{\mathcal{R}}_{n}(q)$ and
$\overline{\mathcal{N}}_{n}(q)$ are verified, among them a $q$-extension of
Newton's inequalities (Theorem \ref{sec:q-newton-ineq}) and $q$-extensions of
the theorems of Rolle (Theorem \ref{sec:q-rolle-thm}) and Laguerre (Theorem
\ref{sec:q-Laguerre-thm}). In Section \ref{sec:line-oper-pres-1} we prove
Theorem \ref{sec:real-zeros-main-thm} and a $q$-extension of Corollary
\ref{sec:introduction-1}(\ref{item:6}) by means of two interspersion invariance
results (Theorems \ref{sec:thm-inv-lhd-vee} and \ref{sec:thm-inv-prec-prec})
concerning the classes $\overline{\mathcal{R}}_{n}(q)$ and
$\overline{\mathcal{N}}_{n}(q)$, which are of independent interest. In
the final Section \ref{sec:proofs-theorems} we present the proof of Lemma
\ref{sec:ext-rusch-lemma}.

\section{Zeros and $n$-Zeros of Polynomials and Rational Functions}
\label{sec:prel-defin}

We consider $\overline{\R}:=\R\cup\{\infty\}$ as being diffeomorphic to the unit
circle $\mathbb{T}:=\{z\in\C:|z|=1\}$ in the Riemann sphere
$\overline{\C}:=\C\cup\{\infty\}$. In that spirit, we use the convention $\pm
\infty:=\infty$ in expressions like $(a,+\infty]$ with $a\in\R$, i.e. if
$b\in(a,+\infty]$ and $b$ is not finite then $b=\infty$. 

A function $F$ that is analytic in a neighborhood of $z\in\C$ has a zero of
order (or multiplicity) $m\in\N_{0}$ at $z$ if $F^{(k)}(z)=0$ for
$k\in\{0,\ldots,m-1\}$ and $F^{(m)}(z)\neq 0$. $\operatorname{ord} (F;z)$ will
denote the order of $z\in\C$ as a zero of $F$. For a polynomial $F$ of degree
$\leq n$ we set
\begin{equation}
  \label{eq:41}
  F^{*n}(z) := z^{n} F\left(-\frac{1}{z}\right).
\end{equation}
Then $F^{*n}$ is a polynomial of degree $\leq n$ and we call $z\in\overline{\C}$
an \emph{$n$-zero} of order $m$ of $F$ and write $\operatorname{ord}_{n}
(F;z)=m$, if $\operatorname{ord} (F;z)=m$ or $\operatorname{ord}
(F^{*n};-1/z)=m$. In this way the number of $n$-zeros of every polynomial $F$ of
degree $m\in\{0,1,\ldots,n\}$ is exactly $n$ (counted according to
multiplicity), since such a polynomial has an $n$-zero of order $n-m$ at
$\infty$.

A rational function $R$ is of degree $n\in\N_{0}$ if $R=F/G$ with polynomials
$F(z)=a_{n}z^{n}+\cdots +a_{0}\nequiv 0$ and $G(z) = b_{n}z^{n} + \cdots +b_{0}
\nequiv 0$ that have no common zeros and for which $n=\max\{\deg F,\deg G\}$. We
extend $R$ to $\overline{\C}$ by letting $R(\infty)$ be equal to $a_{n}/b_{n}$
or $\infty$ depending on whether $b_{n}\neq 0$ or $b_{n}=0$. If $R(\infty)=0$,
then the order of the zero $\infty$ of $R$ is defined to be the order of the
zero of $R(-1/z)$ at the origin (i.e. $\operatorname{ord}
(R;\infty):=\operatorname{ord} (R(-1/z);0)$). In this way every rational
function of degree $n$ has exactly $n$ zeros (counted according to multiplicity)
in $\overline{\C}$. 

If $R$ is a rational function for which $R(\infty)$ is finite, then we set
$R'(\infty):=(R\circ \psi)'(0)$, where $\psi(z):=-1/z$. If $z\in\overline{\R}$
is a pole of $R$, then $R'(z):=(\psi \circ R)'(z)$. One can then see that
if a rational function $R$ has a pole of order $\geq 2$ or a local extremum at
$z\in\overline{\R}$, then $R'(z)=0$, and that if $R'(z)\neq 0$, then there is a
neighborhood $U\subset \overline{\R}$ of $z$ such that $R'(w)R'(z)>0$ for all
$w\in U$.

Finally, for $F:\Omega\subseteq\C\rightarrow \C$ and $y\in\C$, we will use the
notations $F_{\infty}(z):=-zF(z)$ and $F_{y}(z):=F(z)/(z-y)$.

\section{Linear Operators Preserving Interspersion}
\label{sec:polyn-with-intersp}

We say that two polynomials $F$, $G\in\R[z]$ with only real zeros have
\emph{interspersed zeros} if between every pair of successive zeros of $F$ there
is exactly one zero of $G$ (the zeros counted according to
multiplicity). Moreover, we will use the convention that every polynomial
$F\in\R[z]$ with only real zeros and the polynomial $0$ have interspersed
zeros. A particular consequence of this definition is the following.

\begin{lemma}
  \label{sec:polyn-with-intersp-2}
  If $F$, $G\in\R_{n}[z]\setminus\{0\}$ have interspersed zeros and if
  $x\in\overline{\R}$ is an $n$-zero of order $m$ of $F$, then
  $\operatorname{ord}_{n} (G;x)\in\{m-1,m,m+1\}$. In particular, if $F$ and $G$
  have interspersed zeros, then $|\deg F -\deg G|\leq 1$.
\end{lemma}

If $F$, $G\in\pi_{n}(\R)$ have interspersed zeros but no common zeros, then we
say that $F$ and $G$ have \emph{strictly interspersed zeros}. All zeros of two
polynomials with strictly interspersed zeros are simple and we will also say
that polynomial $F\in\R[z]$ with only real simple zeros and the polynomial $0$
have strictly interspersed zeros. 

In \cite[Lem. 1.55, 1.57]{fisk} it is shown that polynomials with interspersed
zeros can be characterized in the following way.

\begin{lemma}
  \label{sec:polyn-with-intersp-1}
  Let $F\in \pi_{n}(\R)\setminus\{0\}$ and $G \in \R_{n}[z]\setminus\{0\}$.
  \begin{enumerate}
  \item\label{item:41} $F$ and $G$ have strictly interspersed zeros if, and only
    if,
    \begin{equation*}
    F'(z) G(z)-F(z)G'(z) \neq 0 \quad\mbox{for all}\quad z\in\R.  
    \end{equation*}
  \item\label{item:42} If $F'(z) G(z)-F(z)G'(z)$ is either non-positive or
    non-negative for all $z\in\R$, then $F$ and $G$ have interspersed zeros and
    $F'(z) G(z)-F(z)G'(z) = 0$ holds only if $z$ is a common zero of $F$ and
    $G$.
  \item \label{item:43} If $F$ and $G$ have interspersed zeros, then $(F/G)'(z)$
    is positive for every $z\in\overline{\R}$ or negative for every such $z$.
  \end{enumerate}
\end{lemma}

For $F$, $G\in\pi_{n}(\R)$ we write $F \preceq G$ if $F'G-FG'\leq 0$ on $\R$ (in
particular $F \preceq 0$ and $0 \preceq F$ for all $F\in\pi_{n}(\R)$), and
$F\prec G$ if $F\preceq G$ and $F$ and $G$ have no common zeros. If
$F\in\sigma_{n}(\R)$ we also write $0\prec F$ and $F\prec 0$. By the above
lemma, $F\preceq G$ implies that $F$ and $G$ have interspersed zeros; moreover,
for $F$,$G\nequiv 0$ we have $F\prec G$ if, and only if, $F'G - FG'<0$ on $\R$.

\begin{lemma}
  \label{sec:lemmas-2}
  Suppose $F$, $G\in\R[z]\setminus\{0\}$ satisfy $F\preceq G$ and
  $F\neq_{\R}G$. Then for all $s\in\R$ we have $F+sG \preceq G$ and $F\preceq
  G+sF$. Furthermore, for $s,t\in\R$ we have $F+sG\preceq F+tG$ and $F+sG
  \neq_{\R} F+tG$ if, and only if, $s<t$.

  On the other hand, if $F$, $G\in\R[z]\setminus\{0\}$ and if there are
  $s,t\in\R$ with $s<t$ such that $F+sG\preceq G$, $F\preceq G+sF$, or
  $F+sG\preceq F+tG$, then $F\preceq G$.

  These statements remain true if we replace $\preceq$ by $\prec$ everywhere.  
\end{lemma}

\begin{proof}
  The assertions follow readily from the relations
  \begin{equation*}
    \left(\frac{F+sG}{F+tG}\right)' = \left(t-s\right)
    \frac{F'G-FG'}{\left(F+tG\right)^{2}},\quad 
    \left(\frac{F+sG}{G}\right)' = \left(\frac{F}{G}\right)',\quad 
    \left(\frac{G+sF}{F}\right)' = \left(\frac{G}{F}\right)',
  \end{equation*}
  and Lemma \ref{sec:polyn-with-intersp-1}.
\end{proof}

The following characterization of interspersion is the essential ingredient of
our proofs of Theorems \ref{sec:real-zeros-main-thm} and
\ref{sec:ext-rusch-lemma}. It is more or less equal to \cite[Lemma
7]{garwag96}, but, as explained in \cite{wag91}, it seems to have been known for
a long time. To some extent, it can also be found in \cite{pol18}, for
example. For the sake of completeness, we present a proof of it here.

\begin{lemma}
  \label{sec:main-lemma}
  Let $F$ be a polynomial of degree $n\in\N_{0}$ that has only real and simple
  zeros $y_{1},\ldots,y_{n}$.
  \begin{enumerate}
  \item \label{item:44} For every polynomial $G\in\R_{n+1}[z]\setminus\{0\}$
    there are $c_{\infty},c_{0},c_{y_{1}},\ldots,c_{y_{n}}\in\R$ such that
    \begin{equation}
      \label{eq:35}
      G(z) =  c_{\infty}F_{\infty}(z) + c_{0}F(z) + \sum_{k=1}^{n} c_{y_{k}} 
      F_{y_{k}}(z),
    \end{equation}
    where for every $n+1$-zero $y$ of $F$ we have 
    \begin{equation}
      \label{eq:36}
      c_{y} = \frac{1}{\left(\frac{F}{G}\right)'(y)}
    \end{equation}
    if $y$ is not an $n+1$-zero of $G$ and $c_{y}=0$ if $y$ is an $n+1$-zero of
    $G$. In particular, $c_{\infty}=0$ if, and only if, $\deg G\leq n$.
  \item \label{item:45} For a polynomial $G\in\R_{n+1}[z]\setminus\{0\}$ we have
    $F\preceq G$ if, and only if, in the representation (\ref{eq:35}) of $G$,
    $c_{y}\leq 0$ for all $n+1$-zeros $y$ of $F$. $F\prec G$ holds if, and only
    if, $c_{\infty}\leq 0$ and $c_{y}< 0$ for all zeros $y$ of $F$.
  \end{enumerate}
\end{lemma}

\begin{proof}
  Partial fraction decomposition of $G/F$ shows that there are
  $c_{\infty},c_{0},c_{y_{k}}\in\R$, $k\in\{1,\ldots,n\}$, with $c_{\infty}\neq
  0$ if, and only if, $\deg G =n+1$ and $c_{y_{k}}\neq 0$ if, and only if,
  $G(y_{k})\neq 0$ such that
  \begin{equation*}
    \frac{G(z)}{F(z)} = -c_{\infty} z + c_{0} + 
    \sum_{k=1}^{n} \frac{c_{y_{k}}}{z-y_{k}}.
  \end{equation*}
  If $\deg G = n+1$ we have $R(0)=0$ and $R'(0)\neq 0$ for $R(z):=
  F(-1/z)/G(-1/z)$ and therefore
  \begin{equation*}
    c_{\infty} = \lim_{z\rightarrow\infty}\frac{1}{-\frac{zF(z)}{G(z)}} = 
    \lim_{z\rightarrow 0}\frac{z}{R(z)-R(0)} = \frac{1}{R'(0)} =  
    \frac{1}{\left(\frac{F}{G}\right)'(\infty)}.
  \end{equation*}
  Similarly it follows that 
  \begin{equation*}
    c_{y_{k}} =  \lim_{z\rightarrow y_{k}}
    \frac{G(z)(z-y_{k})}{F(z)} = \frac{1}{\frac{F'(y_{k})}{G(y_{k})}}
    = \frac{1}{\left(\frac{F}{G}\right)'(y_{k})}
  \end{equation*}
  for every zero $y_{k}$ of $F$ with $G(y_{k})\neq 0$. This proves
  (\ref{item:44}).

  If $F\preceq G$, then by Lemma \ref{sec:polyn-with-intersp-1}(\ref{item:43})
  and the definition of $\preceq$ we have $(F/G)'(z)<0$ for all
  $z\in\overline{\R}$. This implies $c_{y}<0$ for every $n+1$-zero of $F$ that
  is not an $n+1$-zero of $G$ by (\ref{eq:36}). 
  
  On the other hand, if $c_{y}<0$ for every $n+1$-zero of $F$ that is not an
  $n+1$-zero of $G$, then (\ref{eq:36}) shows that $(F/G)'(y)<0$ for every zero
  $y$ of $F/G$. Consequently, $F/G$ has to have a pole between every pair of
  consecutive zeros (recall that we consider $\overline{\R}$ to be
  circular). This shows that $F$ and $G$ have interspersed zeros and, since
  $(F/G)'(y)<0$ at the zeros $y$ of $F/G$, we have $F\preceq G$ by Lemma
  \ref{sec:polyn-with-intersp-1}(\ref{item:43}).
\end{proof}

\begin{corollary}
  \label{sec:polyn-with-intersp-3}
  Let $F\in\sigma_{n}(\R)\setminus\{0\}$ and suppose
  $L:\R_{n}[z]\rightarrow\R[z]$ is a real linear operator. Denote the set of
  $n$-zeros of $F$ by $\mathcal{Z}$.
  \begin{enumerate}
  \item \label{item:46} Suppose $x\in\overline{\R}$, $m\in\N$, and $k\in\N_{0}$,
    are such that 
    \begin{equation*}
      \operatorname{ord}_{m} (L[F];x) \geq k,\quad\mbox{and} \quad
      \operatorname{ord}_{m}(L[F_{y}];x) 
      \geq k \quad \mbox{for all}\quad y\in\mathcal{Z}.
    \end{equation*}
    Then $\operatorname{ord}_{m} (L[G];x)\geq k$ for all
    $G\in\R_{n}[z]\setminus\{0\}$.
  \item \label{item:47} If
    \begin{equation}
      \label{eq:37}
      L\left[F_{y}\right]\preceq  L\left[F\right].
    \end{equation}
    for all $y\in\mathcal{Z}$, then, setting $m:= 1+\deg L[F]$, we have
    $\operatorname{ord}_{m} (L[G];x)\geq -1 + \operatorname{ord}_{m} (L[F];x)$
    for all $G\in\R_{n}[z]\setminus\{0\}$ and all $m$-zeros $x$ of $L[F]$.
  \end{enumerate}
\end{corollary}
\begin{proof}
  Let $x$, $m$, and $k$ be as described in (\ref{item:46}). Suppose first that
  $x\in\R$. Then our assumptions imply that, for all
  $G\in\R_{n}[z]\setminus\{0\}$, $x$ is a zero of order $k$ of $L[G]$, since by
  Lemma \ref{sec:main-lemma}(\ref{item:44})
  \begin{equation}
    \label{eq:38}
    L[G] =  c_{\infty}L[F_{\infty}]  + c_{0}L[F] + \sum_{k=1}^{n} c_{y_{k}} 
    L[F_{y_{k}}].
  \end{equation}
  If $x=\infty$, then we have to show that $L[G]$ is of degree $\leq m-k$. Since
  our assumptions in this case imply $\deg L[F]$, $\deg L[F_{y}]\leq m-k$ for
  all $n$-zeros $y$ of $F$, the assertion follows again from (\ref{eq:38}). This
  proves (\ref{item:46}).

  Because of Lemma \ref{sec:polyn-with-intersp-2}, (\ref{item:47}) follows from
  (\ref{item:46}).
\end{proof}

\begin{lemma}
  \label{sec:working}
  Let $L:\R_{n}[z]\rightarrow\R[z]$ be a real linear operator and suppose $F$,
  $G\in\R_{n}[z]\setminus\{0\}$ satisfy $F\preceq G$. Suppose
  $x^{*}\in\overline{\R}$, for some $m\in\N$, is a simple $m$-zero of $L[F]$. If
  there is at least one zero $y\in\overline{\R}$ of $F/G$ with
  $(L[F]/L[F_{y}])(x^{*})=0$ and if, for all such zeros $y$ of $F/G$,
  \begin{equation}
    \label{eq:39}
    \left(\frac{L[F]}{L[F_{y}]}\right)'(x^{*})> 0,
  \end{equation}
  then $(L[F]/L[G])(x^{*})=0$ and
  \begin{equation}
    \label{eq:44}
     \left(\frac{L[F]}{L[G]}\right)'(x^{*})<0. 
  \end{equation}
\end{lemma}

\begin{proof}
  By considering, in the case $x^{*}=\infty$, the linear operator $H\mapsto
  (L[H])^{*m}$, $H\in \R_{n}[z]$, instead of $L$, we can assume that
  $x^{*}\in\R$.

  Let $\mathcal{Z}$ denote the set of zeros $y$ of $F/G$ in $\overline{\R}$ with
  $L[F_{y}](x^{*})\neq 0$. Suppose that $\mathcal{Z}$ is not empty and that
  (\ref{eq:39}) holds for all $y\in\mathcal{Z}$. Then, by Lemma
  \ref{sec:main-lemma}(\ref{item:45}), for every $y\in \mathcal{Z}$ there is a
  $c_{y} < 0$ such that
  \begin{equation*}
    \frac{L[G](x^{*})}{L[F]'(x^{*})} =  \sum_{y\in \mathcal{Z}}
    c_{y} \frac{L[F_{y}](x^{*})}{L[F]'(x^{*})}  
    = \sum_{y\in \mathcal{Z}} c_{y} \frac{1}{\frac{L[F]'(x^{*})}{L[F_{y}](x^{*})}}
    =  \sum_{y\in \mathcal{Z}} c_{y}
    \frac{1}{\left(\frac{L[F]}{L[F_{y}]}\right)'(x^{*})}.
  \end{equation*}
  Because of (\ref{eq:39}) this means that 
  \begin{equation}
    \label{eq:4}
    \frac{L[G](x^{*})}{L[F]'(x^{*})} <0.
  \end{equation}
  This implies $L[G](x^{*})\neq 0$ and hence that $(L[F]/L[G])(x^{*})=0$. Since,
  moreover,
  \begin{equation*}
    \left(\frac{L[F]}{L[G]}\right)'(x^{*}) =
    \frac{L[F]'(x^{*})}{L[G](x^{*})}= \frac{1}{\frac{L[G](x^{*})}{L[F]'(x^{*})}}, 
  \end{equation*}
  (\ref{eq:4}) also implies (\ref{eq:44}).
\end{proof}

The following consequence of Lemma \ref{sec:main-lemma} can be seen as the real
polynomial analogue of Ruscheweyh's convolution lemma (i.e. of Lemma
\ref{sec:rusch-lemma}). It will play a crucial role in our proof of Theorem
\ref{sec:real-zeros-main-thm}. 

\begin{lemma}
  \label{sec:main-results-2}
  Let $L:\R_{n}[z]\rightarrow\R[z]$ be a real linear operator and suppose $F$,
  $G\in\R_{n}[z]\setminus\{0\}$ are such that $F/G=P/Q$ with $P$,
  $Q\in\sigma_{n}(\R)$ that satisfy $P\prec Q$. Denote the set of zeros of $F/G$ in
  $\overline{\R}$ by $\mathcal{Z}$. If for every $y\in\mathcal{Z}$ we have
  \begin{equation}
    \label{eq:40}
    L\left[F_{y}\right]\preceq  L\left[F\right],
  \end{equation}
  then $L[F]\preceq L[G]$. If there is one $y\in\mathcal{Z}$ for which
  (\ref{eq:40}) holds with $\preceq$ replaced by $\prec$, then
  $L[F]\prec L[G]$.
\end{lemma}

\begin{proof}
  If $F$ and $G$ have a non-constant greatest common divisor $C$, then we
  consider the linear operator $H \mapsto L[HC]$, $H\in\R_{n-\deg C}[z]$,
  instead of $L$. We can therefore assume that $F=P$ and $G=Q$, i.e. that $F$
  and $G$ belong to $\sigma_{n}(\R)$ and satisfy $F\prec G$. $\mathcal{Z}$ is
  then equal to the set of $n$-zeros of $F$.

  Let $D$ denote the greatest common divisor of all polynomials $L[H]$,
  $H\in\R_{n}[z]$. If there is an $x\in\R$ that is a zero of order $\geq k\in\N$
  of $L[F]$ and all $L[F_{y}]$, $y\in\mathcal{Z}$, then, by Corollary
  \ref{sec:polyn-with-intersp-3}(\ref{item:46}), the polynomial $(z-x)^{k}$ is a
  factor of $D$. Moreover, because (\ref{eq:40}) holds for all
  $y\in\mathcal{Z}$, Corollary \ref{sec:polyn-with-intersp-3}(\ref{item:47})
  shows that if $x\in\R$ is a zero of order $k\in\N$ of $L[F]$, then
  $(z-x)^{k-1}$ divides $D$. Hence, by considering the linear operator $H
  \mapsto L[H]/D$, $H\in\R_{n}[z]$ instead of $L$, we can assume that $L[F]$ has
  only simple zeros and that for every zero $x$ of $L[F]$ there is at least one
  $y\in\mathcal{Z}$ such that $(L[F]/L[F_{y}])(x)=0$. Corollary
  \ref{sec:polyn-with-intersp-3}(\ref{item:47}) and (\ref{eq:40}) also show that
  we can assume that $L:\R_{n}[z]\rightarrow \R_{m}[z]$ with $m=\deg L[F]$ or
  $m=1+\deg L[F]$, where, in the latter case, there is at least one
  $y\in\mathcal{Z}$ with $\deg L[F_{y}] = 1 + \deg L[F]$.

  Under these assumptions, Lemma \ref{sec:polyn-with-intersp-1}, (\ref{eq:40}),
  and the definition of $\preceq$, yield that for every $x$ in the set
  $\mathcal{X}$ of $m$-zeros of $L[F]$ there is at least one $y\in \mathcal{Z}$
  with $\left(L[F]/L[F_{y}]\right)'(x)\neq 0$, and that
  \begin{equation*}
    \left(\frac{L[F]}{L[F_{y}]}\right)'(x)>0
  \end{equation*}
  for all such $y$. Because of Lemma \ref{sec:working} this implies
  \begin{equation}
    \label{eq:11}
    \left(\frac{L[F]}{L[G]}\right)'(x)<0  \quad \mbox{for all} \quad 
    x\in\mathcal{X}.
  \end{equation}
  Hence, if 
  \begin{equation*}
    L[G](z) =   c_{0}L[F](z) + \sum_{x\in\mathcal{X}} c_{x}  L[F]_{x}(z)
  \end{equation*}
  is the representation of $L[G]$ in terms of the polynomials $L[F]$ and
  $L[F]_{x}$, $x\in\mathcal{X}$ (given by Lemma
  \ref{sec:main-lemma}(\ref{item:44})), then (\ref{eq:36}) and (\ref{eq:11})
  show that $c_{x}< 0$ for all $x\in\mathcal{X}$. By Lemma
  \ref{sec:main-lemma}(\ref{item:45}) this implies $L[F] \prec L[G]$, as
  required.

  What we have shown now also proves that we have $L[F]\prec L[G]$ if
  (\ref{eq:40}) holds with $\preceq$ replaced by $\prec$ for one
  $y_{0}\in\mathcal{Z}$. For, in such a case all zeros of $L[F]$ are simple and
  $L[F_{y_{0}}](x)\neq 0$ for every zero $x$ of $L[F]$ and thus the greatest
  common divisor $D$ considered above must be a constant.
\end{proof}

\section{Polynomials with Log-Interspersed Zeros}
\label{sec:polynomials-with-log}

It is obvious that for $q\in(0,1)$ a polynomial $F\in\pi_{n}(\R_{0}^{-})$
belongs to $\overline{\mathcal{N}}_{n}(q)$ if, and only if, $F(z)$ and
$F(q^{-1}z)$ have interspersed zeros. It is also clear, however, that for no
$F\in\overline{\mathcal{R}}_{n}(q)$ that has both positive and negative zeros
the polynomials $F(z)$ and $F(q^{-1}z)$ have interspersed zeros. We therefore
need to extend the notion of interspersion in order to characterize all
polynomials $F\in\mathcal{R}_{n}(q)$ in terms of the zero locations of $F(z)$
and $F(q^{-1}z)$.

For $F$, $G\in\pi_{n}(\R)$ we write $F \veebar G$ if $G \preceq zF$ (in
particular $0\veebar F$ and $F \veebar 0$ for all $F\in\pi_{n}(\R)$). Moreover,
we write $F \vee G$ if $F\veebar G$ and $F$ and $G$ have no common zeros expect
possibly a common zero at the origin. Hence, if $F \prec_{0} G$ is supposed to
mean that $F \preceq G$ and that $F$ and $G$ have no common zeros except
possibly a common zero at the origin, then we have $F\vee G$ if, and only if, $G
\prec_{0} zF$. We shall also use the conventions $0 \vee F$, $F\vee 0$, $0
\prec_{0} F$, and $F\prec_{0} 0$, for all polynomials $F\in\pi_{n}(\R)$ which
have a multiple zero at most at the origin.

It is easy
to see that $F \veebar G$ implies $\deg F \leq \deg G \leq 2 +\deg F$.  We say
that $G$ \emph{log-intersperses} $F$ if $F \veebar G$ or $F\veebar -G$. $G$
\emph{strictly log-intersperses} $F$ if $G$ log-intersperses $F$ but $F$ and $G$
have no common zeros except possibly a common zero at the origin.

The next lemma gives a characterization of the relation $F\veebar G$ in terms of
the zeros of $F$ and $G$.

\begin{lemma}
  \label{sec:polynomials-with-log-1}
  Let $F$ and $G$ be two polynomials with only real zeros for which $F/G$ is a
  rational function of degree $n\in\N$. Denote by 
  \begin{equation*}
    -\infty\leq x_{1} \leq x_{2} \leq \cdots \leq x_{n-1} \leq x_{n}\leq +\infty 
    \quad \mbox{and} \quad
    -\infty\leq y_{1} \leq y_{2} \leq \cdots \leq y_{n-1} \leq y_{n}\leq +\infty,
  \end{equation*}
  respectively, the zeros and poles of $F/G$ in $\overline{\R}$ (counted
  according to multiplicity). Then $F \veebar G$ if, and only if, there is a
  $k\in\{0,\ldots,n\}$ such that
  \begin{equation}
    \label{eq:9}
    -\infty\leq x_{1}<y_{1} < x_{2} < y_{2} < x_{3} < \cdots < y_{k-1} <  x_{k} <
    y_{k}\leq 0,
  \end{equation}
  \begin{equation}
    \label{eq:10}
    0\leq y_{k+1} < x_{k+1} < y_{k+1} < x_{k+2} < \cdots  < y_{n-1} < x_{n-1} <
    y_{n}<x_{n}\leq +\infty, 
  \end{equation}
  and, in the case $y_{k}=y_{k+1}=0$,
  \begin{equation}
    \label{eq:27}
    \lim_{z \rightarrow 0,z\in\R} \frac{F(z)}{G(z)} = -\infty,
  \end{equation}
  or, in the case $y_{k}<y_{k+1}$,
  \begin{equation}
    \label{eq:29}
    \frac{F(z)}{G(z)}>0 \quad \mbox{for at least one} \quad z\in(y_{k},y_{k+1}).
  \end{equation}
  In fact, if $F\veebar G$, then $(F/G)(z)>0$ for all $z\in(y_{k},y_{k+1})$.
\end{lemma}

\begin{proof}
  If (\ref{eq:9}) and (\ref{eq:10}) hold, then it is clear that $zF$ and $G$
  have interspersed zeros and it only remains to show that $zF/G$ is
  increasing at some point in $\R$. If $y_{k}=y_{k+1}=0$, then $F/G$ has a
  double pole at the origin and hence (\ref{eq:27}) implies that
  \begin{equation}
    \label{eq:12}
    \lim_{z \rightarrow y_{k}^{-}} \frac{zF(z)}{G(z)} = +\infty \quad
    \mbox{and}\quad 
    \lim_{z \rightarrow y_{k+1}^{+}} \frac{zF(z)}{G(z)} = -\infty.
  \end{equation}
  Consequently, $zF/G$ is increasing around $0$.  If $y_{k} < y_{k+1}$, then by
  (\ref{eq:9}) and (\ref{eq:10}) $F/G$ neither vanishes nor has a pole in
  $(y_{k},y_{k+1})$. If $0<y_{k+1}<+\infty$ it therefore follows from
  (\ref{eq:29}) that
  \begin{equation*}
    \lim_{z \rightarrow y_{k+1}^{-}} \frac{F(z)}{G(z)} = +\infty.
  \end{equation*}
  This implies $\lim_{z \rightarrow y_{k+1}^{-}} zF(z)/G(z) = +\infty$. If
  $y_{k+1}=+\infty$, i.e. if $k=n$, then it follows from (\ref{eq:9}),
  (\ref{eq:10}), (\ref{eq:29}) that $y_{1}\leq 0$ and that $F/G$ is positive in
  $(y_{n},x_{1})$ (recall that we consider $\overline{\R}$ to be circular) and
  negative in $(x_{1},y_{1})$. Consequently, $\lim_{z \rightarrow y_{1}^{-}}
  F(z)/G(z) = -\infty$ has to hold. This implies $\lim_{z \rightarrow y_{1}^{-}}
  zF(z)/G(z) = +\infty$ in the case $y_{1}<0$. In the case $y_{1}=0$ it follows
  from (\ref{eq:9}), (\ref{eq:10}), (\ref{eq:29}) that $k=n=1$ and that
  $F/G=1/z-1/x_{1}$ with $x_{1}\in[-\infty,0)$. Thus, in all possible cases,
  $zF/G$ is increasing at some point in $\R$. This shows the assertion in the
  case $y_{k}\leq 0 < y_{k+1} \leq +\infty$ and in a similar way one can prove
  that $zF/G$ is increasing at some point in $\R$ if $y_{k} \in[-\infty,0)$.

  If, on the other hand, $F\veebar G$, then $G \preceq zF$. Hence, the zeros and
  poles of $zF/G$ lie interspersed on the real line and therefore it is clear
  that the zeros $x_{j}$ and poles $y_{j}$ of $F/G$ must satisfy (\ref{eq:9})
  and (\ref{eq:10}) for a certain $k\in\{0,\ldots,n\}$. If $y_{k} = y_{k+1} =
  0$, then $zF/G$ has a simple pole at $0$ and is increasing around $0$ (since
  $G \preceq zF$). Therefore (\ref{eq:12}) must hold which implies
  (\ref{eq:27}). If $0 < y_{k+1}<+\infty$, then $zF/G \rightarrow +\infty$ as
  $z\rightarrow y_{k+1}^{-}$, since $zF/G$ is increasing in
  $(y_{k},y_{k+1})$. Consequently, $F/G$ must be positive for all
  $z\in(y_{k},y_{k+1})$. If $y_{k+1}=+\infty$ and $y_{k}<0$, then we have $zF/G
  \rightarrow -\infty$ as $z\rightarrow y_{k}^{+}$ and thus that $zF/G$ is
  negative in $(y_{k},0)$ and positive in $(0,y_{k+1})$ (observe that $zF/G$
  vanishes not only at the $x_{j}$ but also at $0$). Hence, $F/G$ is positive in
  $(y_{k},y_{k+1})$. If $y_{k+1}=+\infty$ and $y_{k}=0$, then $zF/G$ is
  increasing and positive in $(x_{k},y_{k+1})$ which implies that $F/G$ is
  positive in $(0,y_{k+1})$.
\end{proof}

In the following, we will write $F\unlhd G$ if $F$, $G\in\pi_{n}(\R_{0}^{-})$
and $F\veebar G$ holds, and $F\lhd G$ if $F$, $G\in\pi_{n}(\R_{0}^{-})$
and $F\vee G$. The preceding lemma shows that the following is true.

\begin{lemma}
  \label{sec:charact-lhd-vee}
  For $F$, $G\in\pi_{n}(\R_{0}^{-})$ we have $F \unlhd G$ or $F\lhd G$ if, and
  only if, $(F/G)(z) >0$ for at least one $z>0$ and, respectively, $F \preceq G$
  or $F \prec_{0} G$.

  In particular, if $F\nequiv 0 \nequiv G$, then $F\unlhd G$ implies $0\leq
  \deg G - \deg F \leq 1$, $0\leq \operatorname{ord} (G;0) -
  \operatorname{ord} (F;0) \leq 1$, and $(F/G)(z) >0$ for all $z>0$.
\end{lemma}

We will need analogues of certain statements regarding polynomials with
interspersed zeros for polynomials with log-interspersed zeros.

First, note that, to some extent, the first direction of Lemma
\ref{sec:lemmas-2} also holds for polynomials with log-interspersed zeros. We
will show the following two lemmas in this respect (it is possible to prove
more complete results, but verifying them seems to be quite straightforward
and they will not be needed in the sequel).

\begin{lemma}
  \label{sec:analogue-s-t-lemma-1}
  Suppose $F$, $G\in\R_{n}[z]$ satisfy $F\vee z G$. Then for every $y\leq 0$ we
  have $F \vee (zG - y F)$.
\end{lemma}
\begin{proof}
  We can assume that $F\nequiv 0 \nequiv G$. $F\vee z G$ implies $z G \prec_{0}
  z F$ and thus $G \prec_{0} F$. By definition this means that $(G/F)'(z)< 0$
  for all $z\in\R$. Hence,
  \begin{equation*}
    \left(\frac{zG-y F}{zF}\right)' = \left(\frac{G}{F}\right)' + 
    \frac{y}{z^{2}} < 0,
  \end{equation*}
  for all $y\leq 0$ and $z\in \R$. Because of Lemma
  \ref{sec:polyn-with-intersp-1} this implies $z G -y F \preceq zF$ and
  therefore that $F \veebar zG - y F$. Since $F \vee zG$, it is easy to see that
  $F(z)=0=z G(z) - y F(z)$ implies $z=0$.
\end{proof}

\begin{lemma}
  \label{sec:polynomials-with-log-6}
  Suppose $F$, $G\in\R_{n}[z]$ are of degree $n\in\N$,
  non-vanishing at $0$, and satisfy $F\vee G$. Set $\alpha:=(F/G)(0)$ and
  $\beta:=(F/G)(\infty)$. Then $(F-\alpha G)/z$ and $F-\beta G$ lie in
  $\R_{n-1}[z]$ and have strictly interspersed zeros.
\end{lemma}

\begin{proof}
  If $x_{j}$ and $y_{j}$ denote, respectively, the zeros of $F$ and
  $G$, then, by Lemma \ref{sec:polynomials-with-log-1} and our
  assumptions, there is a $k\in\{0,\ldots,n\}$ such that
  \begin{equation*}
    -\infty < x_{1} < y_{1} < \cdots < x_{k} < y_{k} < 0 < y_{k+1} < x_{k+1} <
    \cdots  < y_{n} < x_{n} < +\infty,
  \end{equation*}
  and such that $R:=F/G$ is positive in $(y_{k},y_{k+1})$. 

  If $k=n$, then $F\lhd G$ which implies $F\prec G$ by Lemma
  \ref{sec:charact-lhd-vee}. Consequently, $R$ is strictly decreasing
  in $\R$, and since $\deg F =\deg G$, it therefore follows that
  \begin{equation}
    \label{eq:45}
    \alpha > \beta > 0\quad \mbox{and} \quad R((0,+\infty))=(\beta,\alpha).
  \end{equation}
  Lemma \ref{sec:polyn-with-intersp-1} yields that $F -\alpha G \prec F-\beta
  G$. Since $\deg (F -\alpha G) =n$, $\deg (F -\beta G) =n-1$ and $(F -\alpha
  G)(0)=0$, it follows that $(F-\alpha G)/z$ and $F-\beta G$ have strictly
  interspersed zeros. In a similar way one can see that $(F -\alpha G)/z$ and
  $F-\beta G$ must have strictly interspersed zeros if $k=0$. From now on we can
  therefore assume that $k\in\{1,\ldots,n-1\}$.

  In this case, since all zeros $x_{j}$ and all poles $y_{j}$ of $R$ are simple
  and since $R>0$ in $I_{k}:=(y_{k},y_{k+1})$, $R$ jumps from $-\infty$ to
  $+\infty$ at the points $y_{j}$, $j\in M_{-}:=\{1,\ldots,k\}$, and from
  $+\infty$ to $-\infty$ at the points $y_{j}$, $j\in M_{+}:=\{k+1,\ldots,n\}$,
  when $z$ traverses the real line from $-\infty$ to $+\infty$. Consequently,
  $R$ takes every real value at least once in each of the $n-2$ intervals
  $I_{j}:=(y_{j},y_{j+1})$, $j\in\{1,\ldots,k-1,k+1,\ldots,n-1\}$. Moreover,
  since $R$ is continuous and positive in $I_{k}$ with $R\rightarrow +\infty$ as
  $z\rightarrow y_{k}^{+}$ and $z\rightarrow y_{k+1}^{-}$, $M:=\min_{z\in I_{k}}
  R(z)$ must lie in $(0,\alpha]$ and $R$ must take every value $\geq M$ at least
  twice in $I_{k}$. Setting $I_{0}:= (y_{n},y_{1})$ (recall that we consider
  $\overline{\R}$ to be circular), a similar argument shows that $m:=\max_{z\in
    I_{0}} R(z)$ must lie in $[\beta,+\infty)$ and that $R$ must take every
  value $\leq m$ at least twice in $I_{0}$. Since $R$ can take every real value
  at most $n$ times, this implies that (i) $m<M$ and thus also $\beta < \alpha$,
  (ii) $R$ takes every value $\leq m$ or $\geq M$ exactly once in every interval
  $I_{j}$, $j\in\{1,\ldots,k-1,k+1,\ldots,n-1\}$, (iii) $R$ has exactly one
  local extremum $c$ in $I_{k}$, $c$ is a local minimum, and $R(c)=M$, and (iv)
  $R$ has exactly one local extremum $d$ in $I_{0}$, $d$ is a local maximum, and
  $R(d)=m$. In the following we assume that $d\in[-\infty,x_{1})$ (the case
  $d\in(x_{n},+\infty)$ can be treated in a similar manner). Because of the
  monotonicity of $R$ at its poles $y_{j}$, Statements (i)--(iv) imply that
  \begin{equation*}
    R \quad \mbox{decreases in} \quad \left(R^{-1}(M,+\infty) \cup
    R^{-1}(-\infty,m)\right)\cap \bigcup_{j=1}^{k-1} I_{j}
  \end{equation*}
  and in $(y_{k},c)$ and $(d,y_{1})$, and that 
  \begin{equation*}
    R \quad \mbox{increases in} \quad \left(R^{-1}(M,+\infty) \cup
    R^{-1}(-\infty,m)\right)\cap \bigcup_{j=k+1}^{n-1} I_{j}
  \end{equation*}
  and in $(c,y_{k+1})$ and $(y_{n},d)$. Hence, if $a_{1},\ldots,a_{n}$ and
  $b_{1},\ldots,b_{n}$ denote, respectively, the solutions in $\overline{\R}$ of
  the equations $R=\alpha$ and $R=\beta$ (in ascending order with
  $b_{1}=-\infty$), (i) implies that 
  \begin{equation}
    \label{eq:18}
    -\infty = b_{1} \leq b_{2} < x_{1} < y_{1} < a_{1} < b_{3} <
    x_{2} < y_{2} <  \cdots < a_{k-1} < b_{k+1} < x_{k}
    < y_{k} < a_{k} \leq 0     
  \end{equation}
  and
  \begin{equation}
    \label{eq:19}
    0\leq a_{k+1} < y_{k+1} < x_{k+1} < b_{k+2} < a_{k+2} < y_{k+2} <
    x_{k+2} < \cdots < b_{n} < a_{n} < y_{n} < x_{n} < +\infty,
  \end{equation}
  with either $a_{k}=0 \leq a_{k+1}$ or $a_{k}\leq 0=a_{k+1}$,
  depending on whether $c\geq 0$ or $c\leq 0$. Now, if $a_{k}=0 \leq
  a_{k+1}$ (the other case can be treated analogously), then
  $a_{1},\ldots,a_{k-1},a_{k+1},\ldots,a_{n}$ are the zeros of
  $(F-\alpha G)/z$. Moreover, $b_{2},\ldots,b_{n}$ are the zeros of
  $F-\beta G$ (it may happen that $b_{2}=\infty$ in which case
  $F-\beta G$ is of degree $n-2$ with zeros $b_{3},\ldots,b_{n}$). Since, by
  (\ref{eq:18}) and (\ref{eq:19}),
  \begin{equation*}
    -\infty \leq b_{2} < a_{1} < b_{3} < \cdots <a_{k-1} < b_{k+1} <
    a_{k+1} < b_{k+2} < a_{k+2} < \cdots < b_{n} < a_{n}, 
  \end{equation*}
  the proof is complete.
\end{proof}

Because of (\ref{eq:36}) the next lemma can be seen as an analogue of the
'only-if'-direction of Lemma \ref{sec:main-lemma}(\ref{item:45}).

\begin{lemma}
  \label{sec:polynomials-with-log-2}
  Suppose $F$, $G\in \R_{n}[z]\setminus\{0\}$ satisfy $F \veebar G$. Then
  \begin{equation*}
    \left(\frac{F}{G}\right)'(x)<0, \quad
    \left(\frac{F}{G}\right)'(x)>0, \quad \mbox{and} \quad 
    \left(\frac{G}{F}\right)'(y)>0, \quad
    \left(\frac{G}{F}\right)'(y)<0, 
  \end{equation*}
  for, respectively, every negative and positive zero $x$ of $F/G$, and,
  respectively, every negative and positive zero $y$ of $G/F$.
\end{lemma}

\begin{proof}
  $F \veebar G$ implies $G \preceq zF$ and thus, by definition of $\preceq$, 
  \begin{equation*}
    0< \left(\frac{zF}{G}\right)'(x) =  x\left(\frac{F}{G}\right)'(x) 
    \quad \mbox{and} \quad 
    0> \left(\frac{G}{zF}\right)'(y) =  \frac{1}{y}\left(\frac{G}{F}\right)'(y) 
  \end{equation*}
  for every zero $x\neq 0$ of $F/G$ and every zero $y\neq 0$ of $G/F$. 
\end{proof}

Finally, we will also need the following analogue of Lemma
\ref{sec:main-results-2}.

\begin{lemma}
  \label{sec:polynomials-with-log-3}
  Let $L:\R_{n}[z]\rightarrow\R_{m}[z]$ be real linear and suppose $F$,
  $G\in\R_{n}[z]\setminus\{0\}$ are of degree $n\in\N$, have only real zeros and
  satisfy $F\prec G$. Suppose further that $\deg L[F]=m$, $L[F](0)\neq 0 \neq
  L[G](0)$ and that all zeros of $L[F]$ are real and simple. If for every zero
  $y$ of $F$
  \begin{equation}
    \label{eq:1}
    L\left[F_{y}\right]\vee  L\left[F\right]
  \end{equation}
  and if $(L[F]/L[G])(0)>0$, then $L[F]\vee L[G]$.
\end{lemma}

\begin{proof}
  By assumption all zeros $x_{j}$, $j\in\{1,\ldots,m\}$, of $L[F]$ are real,
  simple, and $\neq 0$. Setting $x_{0}:=-\infty$, $x_{m+1}:=+\infty$, we can
  therefore assume that
  \begin{equation*}
    x_{0}<x_{1} < \cdots < x_{k} < 0 < x_{k+1} < \cdots < x_{m}<x_{m+1}
  \end{equation*}
  for a $k\in\{0,\ldots,m\}$. Because of (\ref{eq:1}) and Lemma
  \ref{sec:polynomials-with-log-2} we have
  \begin{equation*}
    \left(\frac{L[F]}{L[F_{y}]}\right)'(x_{j})>0 \quad \mbox{and} \quad
    \left(\frac{L[F]}{L[F_{y}]}\right)'(x_{j})<0
  \end{equation*}
  for, respectively, $j\in M_{-}:=\{1,\ldots,k\}$ and $j\in
  M_{+}:=\{k+1,\ldots,m\}$, and for all zeros $y$ of $F$. Because of Lemma
  \ref{sec:working} this implies
  \begin{equation}
    \label{eq:21}
    \left(\frac{L[F]}{L[G]}\right)'(x_{j})<0 \quad \mbox{and} \quad
    \left(\frac{L[F]}{L[G]}\right)'(x_{j})>0
  \end{equation}
  for, respectively, $j\in M_{-}$ and $j\in M_{+}$. Consequently, $L[F]/L[G]$
  has to have an odd number of poles in each of the intervals $(x_{j},x_{j+1})$,
  $j\in\{1,\ldots,k-1,k+1,\ldots,m-1\}$. (\ref{eq:21}) also shows that
  $L[G]/L[F]$ jumps from $+\infty$ to $-\infty$ at $x_{j}$ with $j\in M_{-}$ and
  from $-\infty$ to $+\infty$ at $x_{j}$ with $j\in M_{+}$. Hence, there is an
  $\epsilon >0$ such that $L[G]/L[F]$ is negative in $(x_{k},x_{k}+\epsilon)$
  and $(x_{k+1}-\epsilon,x_{k+1})$. Since $(L[G]/L[F])(0) >0$ and
  $0\in(x_{k},x_{k+1})$, it thus follows that in the case $k\in\{1,\ldots,m-1\}$
  the rational function $L[G]/L[F]$ must have at least one zero in each of the
  two intervals $(x_{k},0)$ and $(0,x_{k+1})$. Since $L[G]$ must have an odd
  number of poles in each interval $(x_{j},x_{j+1})$,
  $j\in\{1,\ldots,k-1,k+1,\ldots,m-1\}$, $L[G]$ has exactly $m$ zeros $y_{j}$
  which satisfy
  \begin{equation*}
    x_{1} < y_{1} < x_{2} < \cdots < y_{k-1}< x_{k}<y_{k} < 0 <y_{k+1}<
    x_{k+1} < y_{k+2} < \cdots <x_{m-1}<y_{m}< x_{m}. 
  \end{equation*}
  A similar argumentation shows that if $k=m$ (and analogously in the case
  $k=0$), then $L[G]$ has exactly $m$ zeros $y_{j}$ which satisfy
  \begin{equation*}
    x_{1} < y_{1} < x_{2} < \cdots < y_{m-1}< x_{m}<y_{m} < 0.
  \end{equation*}
  Since $(L[F]/L[G])(0) >0$, Lemma \ref{sec:polynomials-with-log-1} thus yields
  $L[F] \vee L[G]$ in all cases under consideration.
\end{proof}

\section{$q$-Extensions of Newton's Inequalities and the Theorems of Rolle and
  Laguerre}
\label{sec:line-oper-pres-1}

Since $F(z)/F(q^{-1}z)$ takes the positive value $q^{\operatorname{ord} (F;0)}$
at $z=0$, Lemmas \ref{sec:polynomials-with-log-1} and \ref{sec:charact-lhd-vee}
shows that the following characterization of the classes
$\overline{\mathcal{R}}_{n}(q)$ and $\overline{\mathcal{N}}_{n}(q)$ is true.

\begin{lemma}
  \label{sec:working-2}
  Let $q\in(0,1)$ and suppose $F\in\R_n[z]\setminus\{0\}$. 
  \begin{enumerate}
  \item \label{item:9} We have $F\in\overline{\mathcal{R}}_{n}(q)$ if, and only
    if, $F(z) \veebar F(q^{-1}z)$ and $F\in\mathcal{R}_{n}(q)$ if, and only if,
    $F(z) \vee F(q^{-1}z)$.
  \item \label{item:10} We have $F\in\overline{\mathcal{N}}_{n}(q)$ if, and only
    if, $F(z) \unlhd F(q^{-1}z)$ and $F\in\mathcal{N}_{n}(q)$ if, and only if,
    $F(z) \lhd F(q^{-1}z)$.
  \end{enumerate}
\end{lemma}

The next lemma will considerably simplify the proofs of our main results.

\begin{lemma}
  \label{sec:an-extension-polyas-1}
  Let $q\in(0,1]$ and suppose $F\in \R_{n}[z]$.
  \begin{enumerate}
  \item \label{item:48} We have $F\in\overline{\mathcal{R}}_{n}(q)$ if, and only
    if, there is a sequence $\{F_{\nu}\}_{\nu\in\N}$ of polynomials $F_{\nu}\in
    \mathcal{R}_{n}(q)$ with $\deg F_{\nu}= n$ and $F_{\nu}(0)\neq 0$ such that
    $F_{\nu}\rightarrow F$ uniformly on compact subsets of $\C$.

  Moreover, if $G\in\pi_{n}(\R)$ satisfies $F\preceq G$, then we can find a
  sequence $\{G_{\nu}\}_{\nu\in\N}\subset\sigma_{n}(\R)$ with $\deg G_{\nu}=
  n$ and $G_{\nu}(0)\neq 0$ such that $F_{\nu} \prec G_{\nu}$ for $\nu\in\N$
  and $G_{\nu}\rightarrow G$ uniformly on compact subsets of $\C$. 
\item \label{item:49} (\ref{item:48}) also holds if
  $\overline{\mathcal{R}}_{n}(q)$, $\mathcal{R}_{n}(q)$, $\R$, $\preceq$,
  $\prec$, are replaced by, respectively, $\overline{\mathcal{N}}_{n}(q)$,
  $\mathcal{N}_{n}(q)$, $\R_{0}^{-}$, $\unlhd$, $\lhd$.
  \end{enumerate}  
\end{lemma}

\begin{proof}
  The 'if'-direction is clear. We will show the 'only if'-direction only for
  (\ref{item:49}), the proof of (\ref{item:48}) being similar.

  Hence, suppose that $F\in\overline{\mathcal{N}}_{n}(q)\setminus\{0\}$ and
  $G\in\pi_{n}(\R_{0}^{-})\setminus\{0\}$ satisfy $F\unlhd G$.  Assume
  further that $F$ is of degree $m\leq n$ with
  $\operatorname{ord}(F;0)=:l\geq 0$ such that
  \begin{equation*}
    F(z) = z^{l} H(z) \quad \mbox{with}\quad H\in 
    \overline{\mathcal{N}}_{m-l}(q), \; \deg H = m-l,\; H(0)\neq 0.
  \end{equation*}
  Suppose $w_{m-l}\leq w_{m-l-1} \leq \cdots \leq w_{2}\leq w_{1}<0$ and
  $a\in\R$ are such that
  \begin{equation*}
    H(z) = a \prod_{j=1}^{m-l}\left(z-w_{j}\right) 
  \end{equation*}
  and set
  \begin{equation*}
    H_{\nu}(z):=a\prod_{j=1}^{m-l}\left(z-(1+\nu^{-1})^{j-1}w_{j}\right) 
    \quad \mbox{for} \quad \nu\in\N.
  \end{equation*}
  Then $H_{\nu} \in \mathcal{N}_{m-l}(q)$ for $\nu\in\N$. Since
  \begin{equation}
    \label{eq:8}
    (R_{l}((1-\nu^{-1})q;- \nu^{-1}z))^{*l}=\prod_{j=1}^{l}(z+\nu^{-1}
    (1-\nu^{-1})^{j-1}q^{j-1})\in \mathcal{N}_{l}(q)
  \end{equation}
  for $\nu\in \N$ and $q\in[0,1]$, we therefore find that, for large $\nu\in\N$,
  \begin{equation*}
    F_{\nu}(z):=(R_{l}((1-\nu^{-1}) q;-\nu^{-1} z))^{*l}\cdot
    H_{\nu}(z)\cdot R_{n-m}((1-\nu^{-1}) q;\nu^{-1} z)
  \end{equation*}
  belongs to $\mathcal{N}_{n}(q)$, is of degree $n$, and does not
  vanish at the origin. (\ref{eq:23}) and (\ref{eq:8}) show that
  \begin{equation*}
    (R_{l}((1-\nu^{-1})q;-\nu^{-1} z))^{*l}\rightarrow z^{l} \quad \mbox{and} \quad
    R_{n-m}((1-\nu^{-1})q;\nu^{-1} z) \rightarrow 1
  \end{equation*}
  as $\nu\rightarrow \infty$, and thus it follows that $F_{\nu}\rightarrow F$
  locally uniformly on $\C$ as $\nu\rightarrow \infty$.

  In the same way one constructs polynomials
  $\hat{G}_{\nu}\in\sigma_{n}(\R_{0}^{-})$ with $\hat{G}_{\nu}(0)\neq 0$ and
  $\deg \hat{G}_{\nu} =n$ that approximate $G$. One can then find a sequence
  $\{s_{\nu}\}_{\nu}\subset(1,+\infty)$ with $s_{\nu}\rightarrow 1$ as
  $\nu\rightarrow \infty$ such that the zeros $y_{k,\nu}$ of
  $G_{\nu}(z):=\hat{G}_{\nu}(s_{\nu}z)$ and the zeros $x_{k,\nu}$ of $F_{\nu}$
  satisfy (\ref{eq:9}). This means that $F_{\nu} \prec G_{\nu}$ and
  $(F_{\nu}/G_{\nu})(z)>0$ for $z>0$. Consequently, $F_{\nu} \lhd G_{\nu}$ for
  $\nu\in\N$ by Lemma \ref{sec:charact-lhd-vee}.
\end{proof}

For $q\in (0,1)$ the \emph{$q$-difference operator $\Delta_{q,n}$} is defined
by
\begin{equation*}
  \Delta_{q,n}[F](z) := \frac{F(z)-F(q^{-1} z)}{q^{n-1} z- q^{-1}z}, \quad
  F\in\R_{n}[z]. 
\end{equation*}
We also set
\begin{equation*}
  \Delta_{q,n}^{*}[F](z) := \frac{q^{-n}F(z)-F(q^{-1} z)}{q^{-n} - 1}, \quad
  F\in\R_{n}[z]. 
\end{equation*}
Using (\ref{eq:33}), it is easy to check that if $F$ is of the
form $F(z) = \sum_{k=0}^{n} C_{k}^{n}(q) a_{k} z^{k}$, then
\begin{equation}
  \label{eq:7}
  \Delta_{q,n}[F](z) = \sum_{k=0}^{n-1} C_{k}^{n-1}(q) a_{k+1} z^{k} \quad
  \mbox{and} \quad
  \Delta_{q,n}^{*}[F](z) = \sum_{k=0}^{n-1} C_{k}^{n-1}(q) a_{k} z^{k}.
\end{equation}
In particular, 
\begin{equation*}
  \Delta_{q,n}[R_{n}(q;z)](z) = \Delta_{q,n}^{*}[R_{n}(q;z)](z)
  = R_{n-1} (q;z).
\end{equation*}
Moreover, if $F(z) = \sum_{k=0}^{n} {n\choose k} a_{k}
z^{k}$, then
\begin{equation*}
  \lim_{q\rightarrow 1} \Delta_{q,n}[F](z) = \lim_{q\rightarrow 1}
  \sum_{k=0}^{n-1} C_{k}^{n-1}(q) \frac{{n\choose k+1}
    a_{k+1}}{C_{k+1}^{n}(q)} z^{k} = \sum_{k=0}^{n-1} {n-1\choose k} a_{k+1}
  z^{k} = \frac{F'(z)}{n},
\end{equation*}
and similarly we see that
\begin{equation*}
  \lim_{q\rightarrow 1} \Delta_{q,n}^{*}[F](z) = F(z)-z\frac{F'(z)}{n}.
\end{equation*}
We therefore set
\begin{equation}
  \label{eq:46}
  \Delta_{1,n}[F](z) := \frac{F'(z)}{n} \quad
  \mbox{and} \quad
  \Delta_{1,n}^{*}[F](z) := F(z)-z\frac{F'(z)}{n}, \quad F\in\R_{n}[z].
\end{equation}

These observations show that $\Delta_{q,n}[F]$ is a $q$-extension of the
derivative $F'$, while $\Delta_{q,n}^{*}[F]$ is a $q$-extension of the polar
derivative of $F$ with respect to $0$ (cf. \cite[(3.1.4)]{rahman}). The next
theorem is therefore a $q$-extension of Rolle's theorem.

\begin{theorem}[$q$-extension of Rolle's theorem]
  \label{sec:q-rolle-thm}
  Let $q\in(0,1]$.
  \begin{enumerate}
  \item \label{item:11} If $F\in \overline{\mathcal{R}}_{n}(q)$, then
    $\Delta_{q,n}[F] \in \overline{\mathcal{R}}_{n-1}(q)$ and $\Delta_{q,n}[F]
    \preceq F$. If $F\in \mathcal{R}_{n}(q)$, then $\Delta_{q,n}[F] \in
    \mathcal{R}_{n-1}(q)$ and $\Delta_{q,n}[F] \prec_{0} F$.
  \item \label{item:12} If $F\in \overline{\mathcal{N}}_{n}(q)$, then
    $\Delta_{q,n}[F] \in \overline{\mathcal{N}}_{n-1}(q)$ and $\Delta_{q,n}[F]
    \unlhd F$. If $F\in \mathcal{N}_{n}(q)$, then $\Delta_{q,n}[F] \in
    \mathcal{N}_{n-1}(q)$ and $\Delta_{q,n}[F] \lhd F$.
  \end{enumerate}
\end{theorem}

\begin{proof}
  We will first verify (\ref{item:11}). For $n=0,1$ the assertions in
  (\ref{item:11}) are trivial and therefore we assume that $n\geq 2$.

  The case $q=1$ is the classical theorem of Rolle together with the
  observation that $nF(z)/(zF'(z))\rightarrow 1$ as $z\rightarrow \infty$
  and that therefore $nF(z)/F'(z)$ has to be increasing for large $z>0$. 
  
  In order to prove the case $q\in(0,1)$, suppose first that $F$ lies in
  $\mathcal{R}_{n}(q)$, is of degree $n$, and satisfies $F(0)\neq 0$. Then
  there is a $k\in\{0,\ldots,n\}$ such that $F$ has $n$ distinct zeros $x_{j}$
  which satisfy
  % ($x_{0}:=-\infty$, $x_{n+1}:=\infty$)
  \begin{equation*}
    x_{1} < x_{2}< \cdots < x_{k} < 0 < x_{k+1} < x_{k+2} < \cdots < x_{n}.
  \end{equation*}
  Because of Rolle's theorem for every $j\in\{1,\ldots,n-1\}$ there is exactly
  one critical point $y_{j}$ of $F$ in $(x_{j}, x_{j+1})$. We will now prove
  the assertion in the case where $T:=F(y_{k})>0$ and $y_{k}>0$ (the other
  possible cases can be verified in a similar manner). 
  
  Under this assumption there are continuous functions $a(t)\in[x_{k},y_{k}]$
  and $b(t)\in[y_{k},x_{k+1}]$ with
  \begin{equation}
    \label{eq:3}
    F(a(t))=F(b(t))=t\qquad \mbox{for}\qquad t\in[0,T].
  \end{equation}
  Consequently, $a(0)=x_{k}$, $b(0)=x_{k+1}$, and $a(T)=b(T)=y_{k}$. Since
  $r(t):=a(t)/b(t)$ is continuous in $[0, T]$ with $r(0)=x_{k}/x_{k+1} \leq 0$
  and $r(T) = 1$ there must be a $t_{0}\in(0,T)$ (actually, $t_{0}\in(F(0),T)$)
  with $r(t_{0}) = q$. Setting $w_{k}:=a(t_{0})$, this means $b(t_{0})=q^{-1}
  w_{k}$ and thus, because of (\ref{eq:3}), $\Delta_{q,n}[F](w_{k}) =
  0$. Moreover,
  \begin{equation}
    \label{eq:6}
    x_{k}< 0 < w_{k} < q^{-1}w_{k} < x_{k+1}.
  \end{equation}

  Since $F\in\mathcal{R}_{n}(q)$ we have $x_{j}/x_{j+1} < q$ for
  $j\in\{k+1,\ldots,n-1\}$ and $j\in\{1,\ldots,k-1\}$.  Making use of these
  inequalities, we can proceed in a similar way as in the case $j=k$ to find
  that $\Delta_{q,n}[F]$ has zeros $w_{j}$,
  $j\in\{1,\ldots,n-1\}\setminus\{k\}$, with
  \begin{align*}
    x_{j}< q^{-1}w_{j} < w_{j} < x_{j+1}, &\quad\mbox{for }j\in\{1,\ldots,k-1\},\\
    x_{j} < w_{j} < q^{-1}w_{j} < x_{j+1}, &\quad \mbox{for
    }j\in\{k+1,\ldots,n-1\}.
  \end{align*}
  This, together with (\ref{eq:6}), shows that
  $\Delta_{q,n}[F]\in\mathcal{R}_{n-1}(q)$ and that $F$ and $\Delta_{q,n}[F]$
  have strictly interspersed zeros. By (\ref{eq:33}) we have
  \begin{equation*}
    \frac{F(z)}{z \Delta_{q,n}[F](z)}\rightarrow
    \frac{C_{k}^{n}(q)}{C_{k-1}^{n-1}(q)}>0 \quad \mbox{as} \quad 
    z\rightarrow\infty.
  \end{equation*}
  Hence, $F(z)/\Delta_{q,n}[F](z)$ is increasing for large $z$, which implies
  $\Delta_{q,n}[F]\prec F$ since $F$ and $\Delta_{q,n}[F]$ have strictly
  interspersed zeros.

  It remains to prove (\ref{item:11}) for $F\in \overline{\mathcal{R}}_{n}(q)$
  or $F\in \mathcal{R}_{n}(q)$ which do not have to be of degree $n$ and $\neq
  0$ at $z=0$. If $F$ is a polynomial in $\overline{\mathcal{R}}_{n}(q)$, then
  Lemma \ref{sec:an-extension-polyas-1} and what we have just shown yield
  $\Delta_{q,n}[F]\in\overline{\mathcal{R}}_{n-1}(q)$ and $\Delta_{q,n}[F]
  \preceq F$. If $\Delta_{q,n}[F]$ and $F$ have a common zero at a point
  $z\in\R\setminus\{0\}$, then it follows that $F(z)=F(q^{-1}z)=0$ and hence
  that $F\notin \mathcal{R}_{n}(q)$. If there is a $z>0$ (in the case $z<0$ one
  can argue analogously) such that
  $\Delta_{q,n}[F](z)=\Delta_{q,n}[F](q^{-1}z)=0$, then
  \begin{equation}
    \label{eq:42}
    F(z)=F(q^{-1} z) = F(q^{-2}z).
  \end{equation}
  This implies that there has to be a zero $x$ of $F$ in $[z,q^{-2}z]$, for
  otherwise $F$ would not vanish in $[z,q^{-2}z]$, but $F'$ would vanish at
  least two times there. Hence, $F$ and $F'$ would not have interspersed zeros,
  a contradiction to the fact that $F$ has only real zeros. We can suppose that
  $x\in[z,q^{-1}z]$. Then because of (\ref{eq:42}) there has to a second zero
  $y$ of $F$ in $[z,q^{-1}z]$. Since $F\in \overline{\mathcal{R}}_{n}(q)$, we
  must have $\{x,y\}=\{z,q^{-1}z\}$ and thus $F\notin \mathcal{R}_{n}(q)$.
  Hence, (\ref{item:11}) is proven. 

  The proof of (\ref{item:11}) also shows that if $F$
  belongs to $\overline{\mathcal{N}}_{n}(q)$ or $\mathcal{N}_{n}(q)$, then
  $\Delta_{q,n}[F]$ belongs to, respectively,
  $\overline{\mathcal{N}}_{n-1}(q)$ or
  $\mathcal{N}_{n-1}(q)$. (\ref{item:12}) thus follows from (\ref{item:11})
  and the definition of $\unlhd$.
\end{proof}

\begin{theorem}
  \label{sec:q-Laguerre-thm-1}
  Let $q\in(0,1]$ and suppose $\mathcal{C}$ denotes one of the classes
  $\overline{\mathcal{R}}_{n}(q)$, $\mathcal{R}_{n}(q)$,
  $\overline{\mathcal{N}}_{n}(q)$, $\mathcal{N}_{n}(q)$. Then $F\in\mathcal{C}$
  implies $\Delta_{q,n}^{*}[F] \in \mathcal{C}$.
\end{theorem}

\begin{proof}
  First, observe that for $q\in(0,1)$ and $F\in\R_{n}[z]$
  \begin{align*}
    \left(\Delta_{q,n}[F]\right)^{*(n-1)}(z) = \; & z^{n-1}
    \frac{F(-z^{-1})-F(-q^{-1} z^{-1})}{-q^{n-1}z^{-1}+ q^{-1}z^{-1}}
    \\
    = \; & -q^{1-n} \frac{q^{-n}F^{*n}(qz)-F^{*n}(z)}{q^{-n}-1} = -q^{1-n}
    \Delta_{q,n}^{*}[F^{*n}] (qz).
  \end{align*}
  It is also straightforward to verify that
  \begin{equation*}
    z\frac{(F^{*n})'(z)}{n} - F^{*n}(z) = \left(\frac{F'(z)}{n}\right)^{*(n-1)}
  \end{equation*}
  and thus, for all $q\in(0,1]$ and $F\in\R_{n}[z]$, we have
  \begin{equation}
    \label{eq:43}
    \left(\Delta_{q,n}[F]\right)^{*(n-1)}(z) = -q^{1-n}\Delta_{q,n}^{*}[F^{*n}](qz).
  \end{equation}
  Together with Theorem \ref{sec:q-rolle-thm}, this relation immediately shows
  that if $F$ belongs to $\overline{\mathcal{R}}_{n}(q)$ or
  $\mathcal{R}_{n}(q)$, then $\Delta_{q,n}^{*}[F]$ is an element of,
  respectively, $\overline{\mathcal{R}}_{n-1}(q)$ or $\mathcal{R}_{n-1}(q)$.

  On the other hand, a polynomial $F$ lies in $\overline{\mathcal{N}}_{n}(q)$ or
  $\mathcal{N}_{n}(q)$ if, and only if, $F^{*n}(-z)$ lies in, respectively,
  $\overline{\mathcal{N}}_{n}(q)$ or $\mathcal{N}_{n}(q)$. Since
  \begin{equation*}
    \Delta_{q,n}[F(-z)](z) = - \Delta_{q,n}[F](-z),
  \end{equation*}
  (\ref{eq:43}) and Theorem \ref{sec:q-rolle-thm} therefore also give the
  remaining parts of the assertion.
\end{proof}

As described in Section \ref{sec:new-char-log} ''Newton's inequalities''
\cite[Thm. 2]{stanley89} state that if $F(z) = \sum_{k=0}^{n} {n\choose k}
a_{k} z^{k}$ is an element of $\overline{\mathcal{R}}_{n}(1)$, then the sequence
$\{a_{k}\}_{k=0}^{n}$ is log-concave. Using Theorems \ref{sec:q-rolle-thm} and
\ref{sec:q-Laguerre-thm-1} the proof of Newton's inequalities (as given in
\cite[Thm. 2]{stanley89}, for example) can be modified in order to obtain the
following.

\begin{theorem}($q$-extension of Newton's inequalities)
  \label{sec:q-newton-ineq}
  Let $q \in (0,1]$ and suppose $F(z) = \sum_{k=0}^{n} C_{k}^{n}(q) a_{k} z^{k}
  \in\mathcal{R}_{n}(q)$. Then $f(z) := \sum_{k=0}^{n} a_{k} z^{k}$ belongs to
  $\mathcal{LC}_{n}$, i.e. we have
  \begin{equation*}
    a_{k}^{2} > a_{k-1} a_{k+1} \quad \mbox{for all} \quad k\in\{0,\ldots,n\}
  \end{equation*}
  for which there are $l\leq k$ and $m\geq k$ such that $a_{l}$, $a_{m}\neq 0$.
  If $F \in\mathcal{N}_{n}(q)$, then $f$ belongs to $\mathcal{LC}_{n}^{+}$,
  i.e. $f\in \mathcal{LC}_{n}$ and all coefficients are either non-positive or
  non-negative.
\end{theorem}

\begin{proof}
  Applying Theorem \ref{sec:q-rolle-thm} $j-1$-times to $F(z) = \sum_{k=0}^{n}
  C_{k}^{n}(q) a_{k} z^{k} \in\mathcal{R}_{n}(q)$ yields that
  \begin{equation*}
    \sum_{k=0}^{n-j+1} C_{k}^{n-j+1}(q) a_{k+j-1} z^{k}  
    \in\mathcal{R}_{n-j+1}(q),
  \end{equation*}
  and applying Theorem \ref{sec:q-Laguerre-thm-1} $n-j-1$-times to this
  polynomial leads to
  \begin{equation*}
    p(z):=q a_{j+1}z^{2} +(1+q) a_{j} z + a_{j-1} = 
    \sum_{k=0}^{2} C_{k}^{2}(q) a_{k+j-1} z^{k}
    \in\mathcal{R}_{2}(q).
  \end{equation*}
  
  By \cite[VIII. Lem. 3]{levin1980} we have $a_{j+1}a_{j-1}<0\leq a_{j}^{2}$ for
  every $j\in\{\operatorname{ord}(F;0)+1,\ldots,-1+\deg F\}$ with $a_{j}=0$. We
  can therefore assume that $a_{j+1}a_{j-1}\neq 0$. Then
  \begin{equation*}
    z_{1,2} := \frac{-(1+q)a_{j}\pm \sqrt{(1+q)^{2} a_{j}^{2}- 4 q a_{j+1}a_{j-1}}}
    {2 q a_{j+1}}
  \end{equation*}
  are the zeros of $p$. It follows that
  \begin{equation*}
    z_{1}z_{2} = \frac{a_{j-1}}{q a_{j+1}}
  \end{equation*}
  and hence that $a_{j+1}a_{j-1} < 0 \leq a_{j}^{2}$ if $z_{1}z_{2} <0$.

  If $z_{1}z_{2}> 0$, then we can assume that $z_{1},z_{2}< 0$ (by considering
  $p(-z)$ instead of $p(z)$ if necessary). Then $a_{j+1},a_{j},a_{j-1}$ must be
  all of same sign and we can assume that they are all positive. Since
  $p\in\mathcal{R}_{2}(q)$, we have $q z_{2} < z_{1}<0$, and hence that
  \begin{equation*}
    q > \frac{4 q a_{j+1}a_{j-1}}
    {\left((1+q)a_{j}+ \sqrt{(1+q)^{2} a_{j}^{2}- 4 q a_{j+1}a_{j-1}}\right)^{2}} = 
    \frac{q}{\left(x+\sqrt{x^{2}-q}\right)^{2}},
  \end{equation*}
  with $x:= (1+q)a_{j}/\sqrt{4 a_{j+1} a_{j-1}}$. This inequality implies
  $\sqrt{x^{2}-q} > 1-x$ and thus
  \begin{equation*}
    \frac{(1+q)a_{j}}{2 \sqrt{a_{j+1} a_{j-1}}}= x > \frac{1+q}{2}.
  \end{equation*}
  This is equivalent to $a_{j}^{2} > a_{j+1}a_{j-1}$ and therefore proves
  the assertion for $F\in\mathcal{R}_{n}(q)$. Since all coefficients of a
  polynomial with only non-positive zeros are of same sign, this also verifies
  the assertion for $F\in\mathcal{N}_{n}(q)$.
\end{proof}

It follows from (\ref{eq:46}) that
\begin{equation*}
  \Delta_{1,n}^{*}[F] + x \Delta_{1,n}[F]
\end{equation*}
is equal (up to a factor $n$) to the polar derivative of a polynomial
$F\in\R_{n}[z]$ with respect to $x$. Laguerre's theorem states that all zeros of
this polar derivative are real if $F\in\pi_{n}(\R)$ and $x\in\R$
(cf. \cite[Thm. 3.2.1]{rahman}). Because of \cite[Satz 5.2]{obresch} this means
that $\Delta_{1,n}[F]$ and $\Delta_{1,n}^{*}[F]$ have interspersed zeros. The
next theorem is therefore a $q$-extension of Laguerre's theorem.

\begin{theorem}[$q$-extension of Laguerre's theorem]
  \label{sec:q-Laguerre-thm}
  Let $q\in(0,1]$. 
  \begin{enumerate}
  \item \label{item:13} If $F\in \overline{\mathcal{R}}_{n}(q)$, then
    $\Delta_{q,n}[F] \preceq \Delta_{q,n}^{*}[F]$. If $F\in \mathcal{R}_{n}(q)$,
    then $\Delta_{q,n}[F] \prec_{0} \Delta_{q,n}^{*}[F]$.
  \item \label{item:14} If $F\in \overline{\mathcal{N}}_{n}(q)$, then
    $\Delta_{q,n}[F] \unlhd \Delta_{q,n}^{*}[F]$. If $F\in
    \mathcal{N}_{n}(q)$, then $\Delta_{q,n}[F] \lhd \Delta_{q,n}^{*}[F]$.
  \end{enumerate}
\end{theorem}

\begin{proof}
  By Theorems \ref{sec:q-rolle-thm} and \ref{sec:q-Laguerre-thm-1} we have
  $\Delta_{q,n}[F]$, $\Delta_{q,n}^{*}[F]\in\pi_{n}(\R_{0}^{-})$ when $F\in
  \overline{\mathcal{N}}_{n}(q)$. Since all coefficients of a polynomial in
  $\pi_{n}(\R_{0}^{-})$ must be of same sign, we also have
  $(\Delta_{q,n}[F]/\Delta_{q,n}^{*}[F])(z)>0$ for all $z>0$ when $F\in
  \overline{\mathcal{N}}_{n}(q)$. (\ref{item:14}) therefore follows directly
  from (\ref{item:13}) and Lemma \ref{sec:charact-lhd-vee}.

  In order to prove (\ref{item:13}), we can assume that $n\geq 2$. We will first
  suppose that $q\in(0,1)$ and that $F$ is an element of
  $\mathcal{R}_{n}(q)$ which is of degree $n$ and does not vanish at the origin.

  Set $R(z):=F(z)/F(q^{-1}z)$. Then $R(0)=1$ and $R(\infty)=q^{n}$. Since $F(z)
  \vee F(q^{-1}z)$ by Lemma \ref{sec:working-2}, it follows from Lemma
  \ref{sec:polynomials-with-log-6} that
  \begin{equation*}
    \frac{F(z)-F(q^{-1}z)}{z}=(q^{n-1}-q^{-1})\Delta_{q,n}[F](z) 
  \end{equation*}
  and
  \begin{equation*}
    F(z)-q^{n}F(q^{-1}z)=(1-q^{n})\Delta_{q,n}^{*}[F](z)
  \end{equation*}
  have strictly interspersed zeros. In order to prove that in fact
  $\Delta_{q,n}[F] \prec \Delta_{q,n}^{*}[F]$, write $F(z)=\sum_{k=0}^{n}
  C_{k}^{n}(q) a_{k} z^{k}$. Then
  \begin{equation*}
    \left(\frac{\Delta_{q,n}[F]}{\Delta_{q,n}^{*}[F]}\right)'(0) =
    C_{1}^{n-1}(q)\frac{a_{0}a_{2}-a_{1}^{2}}{a_{0}^{2}} <0
  \end{equation*}
  by Theorem \ref{sec:q-newton-ineq}. Hence, $\Delta_{q,n}[F]$ and
  $\Delta_{q,n}^{*}[F]$ have interspersed zeros and are decreasing at $0$. This
  implies $\Delta_{q,n}[F]\prec \Delta_{q,n}^{*}[F]$, as required.
  
  By using Lemma \ref{sec:an-extension-polyas-1}, it follows from this that for
  all $q\in(0,1]$ and every $F\in\overline{\mathcal{R}}_{n}(q)$ we have
  $\Delta_{q,n}[F]\preceq \Delta_{q,n}^{*}[F]$. If $\Delta_{q,n}[F]$ and
  $\Delta_{q,n}^{*}[F]$ have a common zero at $z\neq 0$, then necessarily
  $F(z)=F(q^{-1}z)=0$ or $F(z)=F'(z)=0$ (depending on whether $q\in(0,1)$ or
  $q=1$) and thus $F\notin \mathcal{R}_{n}(q)$.
\end{proof}

\section{Weighted Hadamard Products Preserving Zero Interspersion}
\label{sec:weight-hadam-prod}

Because of (\ref{eq:33}) we have $C_{k}^{n}(q)>0$ for all
$k\in\{0,\ldots,n\}$ and $q\in(0,1]$. Consequently, we can write every pair of
polynomials $F$, $G\in\R_{n}[z]$ in the form
\begin{equation*}
  F(z) = \sum_{k=0}^{n} C_{k}^{n}(q) a_{k} z^{k}, \quad
  G(z) = \sum_{k=0}^{n} C_{k}^{n}(q) a_{k} z^{k}, \quad q\in(0,1],
\end{equation*}
which enables us to define
\begin{equation*}
  F*_{q}^{n} G(z):=\sum_{k=0}^{n} C_{k}^{n}(q) a_{k} b_{k} z^{k}.
\end{equation*}
Observe that for $q=1$ the weighted Hadamard product $*_{q}^{n}$ is equal to the
weighted Hadamard product $*_{GS}$ appearing in the Grace-Szeg\"o convolution
theorem.  Note also that if $q\in(0,1]$ and
\begin{equation*}
  H(z) = \sum_{k=0}^{n+1} C_{k}^{n+1}(q) a_{k}
  z^{k} \in \R_{n+1}[z], \quad F(z) = \sum_{k=0}^{n} C_{k}^{n}(q) b_{k}
  z^{k}\in\R_{n}[z],
\end{equation*}
then, using (\ref{eq:7}), it is straightforward to verify that
\begin{equation}
  \label{eq:50}
  \Delta_{q,n+1}^{*}[H]*_{q}^{n} F = H *_{q}^{n+1} F \quad \mbox{and}\quad
  z(\Delta_{q,n+1}[H]*_{q}^{n} F) = H *_{q}^{n+1} z F.
\end{equation}

The following two invariance results concerning the weighted Hadamard product
$*_{q}^{n}$ and the classes $\overline{\mathcal{R}}_{n}(q)$ and
$\overline{\mathcal{N}}_{n}(q)$ are the strongest results in this paper.
 
\begin{theorem}
  \label{sec:thm-inv-lhd-vee}
  Let $q\in(0,1]$ and suppose $H\in \overline{\mathcal{R}}_{n}(q)$ is not
  extremal. Suppose further that $F\in\overline{\mathcal{N}}_{n}(q)$ and
  $G\in\pi_{n}(\R_{0}^{-})$ satisfy $F \unlhd G$ and $F \neq_{\R} G$. Then
  \begin{equation*}
    F*_{q}^{n} H \veebar G*_{q}^{n} H. 
  \end{equation*}
  We have $F*_{q}^{n} H \vee G*_{q}^{n} H$ if $H\in \mathcal{R}_{n}(q)$, $F
  \lhd G$, or if $F$ belongs to $\mathcal{N}_{n}(q)$.
\end{theorem}

\begin{theorem}
  \label{sec:thm-inv-prec-prec}
  Let $q\in(0,1]$ and suppose $H\in \overline{\mathcal{N}}_{n}(q)$ is not
  extremal. Suppose further that $F\in \overline{\mathcal{R}}_{n}(q)$ and
  $G\in\pi_{n}(\R)$ satisfy $F \preceq G$ and $F \neq_{\R} G$. Then
  \begin{equation*}
    F*_{q}^{n} H \preceq G*_{q}^{n} H. 
  \end{equation*}
  We have $F*_{q}^{n} H \prec_{0} G*_{q}^{n} H$ if $H\in \mathcal{N}_{n}(q)$, $F
  \prec_{0} G$, or if $F$ belongs to $\mathcal{R}_{n}(q)$.
\end{theorem}

\begin{proof}[Proof of Theorems \ref{sec:thm-inv-lhd-vee} and
  \ref{sec:thm-inv-prec-prec}]
  The theorems are easy to verify when $n=0$ or $n=1$. Both theorems will
  therefore be proven, if we can show the following two claims for every
  $m\in\N$.

  \textbf{Claim 1:} \textit{If Theorem \ref{sec:thm-inv-lhd-vee} holds for
    $n=m$, then Theorem \ref{sec:thm-inv-prec-prec} holds for $n=m+1$.}

  \textbf{Claim 2:} \textit{If Theorem \ref{sec:thm-inv-prec-prec} holds for
    $n=m$, then Theorem \ref{sec:thm-inv-lhd-vee} holds for $n=m+1$.}

  \textit{Proof of Claim 1.} Let
  $F\in\overline{\mathcal{R}}_{m+1}(q)\setminus\{0\}$,
  $G\in\pi_{m+1}(\R)\setminus\{0\}$ be such that $F\preceq G$ and $F \neq_{\R}
  G$, and suppose that $H\in\overline{\mathcal{N}}_{m+1}(q)\setminus\{0\}$ is not
  extremal. We assume first that $F$, $G$, $H$ do not vanish at the origin and
  are all of degree $m+1$.

  Theorems \ref{sec:q-rolle-thm}, \ref{sec:q-Laguerre-thm-1}, and
  \ref{sec:q-Laguerre-thm}, show that $\Delta_{q,m+1}[H]$ and
  $\Delta_{q,m+1}^{*}[H]$ belong to $\overline{\mathcal{N}}_{m}(q)$ and satisfy
  $\Delta_{q,m+1}[H]\unlhd \Delta_{q,m+1}^{*}[H]$. Theorem
  \ref{sec:thm-inv-lhd-vee} (which holds for $n=m$ by assumption) thus implies
  \begin{equation*}
    \Delta_{q,m+1}[H] *_{q}^{m} F_{y} \veebar  \Delta_{q,m+1}^{*}[H]*_{q}^{m} F_{y},
  \end{equation*}
  or, equivalently,
  \begin{equation*}
    \Delta_{q,m+1}^{*}[H] *_{q}^{m} F_{y} \preceq
    z\left(\Delta_{q,m+1}[H]*_{q}^{m} F_{y}\right) 
  \end{equation*}
  for every zero $y$ of $F$. Because of Lemma \ref{sec:lemmas-2} this means 
  \begin{equation*}
     \Delta_{q,m+1}^{*}[H]
     *_{q}^{m} F_{y} \preceq   z\left(\Delta_{q,m+1}[H]*_{q}^{m} F_{y}\right)
     -y\Delta_{q,m+1}^{*}[H] *_{q}^{m} F_{y} 
  \end{equation*}
  and this, in turn, is equivalent to 
  \begin{equation*}
    H *_{q}^{m+1} F_{y} \preceq H *_{q}^{m+1} zF_{y} - y H *_{q}^{m+1} F_{y} =
    H *_{q}^{m+1} F
  \end{equation*}
  for all zeros $y$ of $F$ by (\ref{eq:50}).  Defining the linear operator
  $L:\R_{m+1}[z] \rightarrow \R_{m+1}[z]$ by $L[P]:=H *_{q}^{m+1} P$ for
  $P\in\R_{m+1}[z]$, we thus obtain $L[F_{y}] \preceq L[F]$ for every zero $y$
  of $F$. Because of Lemma \ref{sec:main-results-2} this means $L[F] \preceq
  L[G]$, which is equivalent to $F *_{q}^{m+1} H \preceq G *_{q}^{m+1} H$.
  
  Applying Lemma \ref{sec:an-extension-polyas-1}, it follows from this special
  case that for every $H\in\overline{\mathcal{N}}_{m+1}(q)$ and all
  $F\in\overline{\mathcal{R}}_{m+1}(q)$, $G\in\pi_{m+1}(\R)$ with $F\preceq G$
  we have $F*_{q}^{m+1} H \preceq G*_{q}^{m+1} H$. Moreover, if $F \prec_{0} G$,
  then there is an $\epsilon_{0}>0$ such that $F(\epsilon z) \preceq G(z)$ for
  all $\epsilon\in(1-\epsilon_{0},1+\epsilon_{0})$. What we have shown so far
  therefore also implies $(F*_{q}^{m+1}H)(\epsilon z) \preceq
  (G*_{q}^{m+1}H)(z)$ for all $\epsilon\in(1-\epsilon_{0},1+\epsilon_{0})$.
  Hence, $F*_{q}^{m+1}H \prec_{0} G*_{q}^{m+1}H$ if $F \prec_{0} G$.

  Next, suppose that $H\in\overline{\mathcal{N}}_{m+1}(q)$ is not extremal and
  that $F\in\overline{\mathcal{R}}_{m+1}(q)$ and $G\in\pi_{m+1}(\R)$ satisfy
  $F\neq _{\R} G$ and $F\preceq G$. It remains to show that under these
  assumptions
  \begin{equation}
    \label{eq:48}
    F *_{q}^{m+1} H \prec_{0} G *_{q}^{m+1} H
  \end{equation}
  if $F\in\mathcal{R}_{m+1}(q)$ or $H\in\mathcal{N}_{m+1}(q)$.

  To that end, denote the set of $m+1$-zeros of $F$ and $G$ by $\mathcal{Z}_{F}$
  and $\mathcal{Z}_{G}$, respectively. Then, since $F_{y} \preceq F$ for every
  $y\in\mathcal{Z}_{F}$, what we have shown so far implies $F_{y}*_{q}^{m+1} H
  \preceq F*_{q}^{m+1} H$, and thus, by Lemma \ref{sec:polyn-with-intersp-1} and
  the definition of $\preceq$,
  \begin{equation}
    \label{eq:49}
    \left(\frac{F*_{q}^{m+1} H}{F_{y}*_{q}^{m+1} H}\right)'(z) > 0 
    \quad \mbox{for all} \quad z\in\R,\; y\in\mathcal{Z}_{F}.
  \end{equation}
  Now observe that, if $F *_{q}^{m+1} H$ and $G *_{q}^{m+1} H$ have a common zero
  $x^{*}\neq 0$, then $x^{*}$ has to be a zero of $F_{y}*_{q}^{m+1} H$ for all
  $y\in\mathcal{Z}_{F}\setminus\mathcal{Z}_{G}$. For otherwise, (\ref{eq:49})
  would hold for all $y\in\mathcal{Z}_{F}\setminus\mathcal{Z}_{G}$ with
  $(F_{y}*_{q}^{m+1} H)(x^{*})\neq 0$ (and there would be at least one such $y$),
  and hence Lemma \ref{sec:working} would imply that $(G*_{q}^{m+1} H)(x^{*})\neq
  0$. Consequently, there is at least one $y\in\mathcal{Z}_{F}$ with
  \begin{equation}
    \label{eq:51}
    (F_{y}*_{q}^{m+1} H)(x^{*})= 0 = (F*_{q}^{m+1} H)(x^{*})
  \end{equation}
  
  If $y\in\R$, then, because of (\ref{eq:50}), we have
  \begin{equation*}
   F_{y}*_{q}^{m+1} H =  \Delta_{q,m+1}^{*}[H] *_{q}^{m} F_{y} 
  \end{equation*}
  and
  \begin{equation*}
    F*_{q}^{m+1} H = (z-y)F_{y}*_{q}^{m+1} H =  
    z\left(\Delta_{q,m+1}[H]*_{q}^{m} F_{y}\right) -
    y \Delta_{q,m+1}^{*}[H] *_{q}^{m} F_{y},  
  \end{equation*}
  and thus (\ref{eq:51}) implies that 
  \begin{equation}
    \label{eq:2}
    (\Delta_{q,m+1}[H] *_{q}^{m} F_{y})(x^{*}) = 0 = 
    (\Delta_{q,m+1}^{*}[H]*_{q}^{m} F_{y})(x^{*}).  
  \end{equation}

  If $m=1$, this means $\Delta_{q,m+1}[H] *_{q}^{m} F_{y} =_{\R}
  \Delta_{q,m+1}^{*}[H]*_{q}^{m} F_{y}$ and consequently $\Delta_{q,m+1}[H]
  =_{\R} \Delta_{q,m+1}^{*}[H]$. Because of Theorem \ref{sec:q-Laguerre-thm}
  this yields $H\in\overline{\mathcal{R}}_{2}(q)\setminus \mathcal{R}_{2}(q)$
  and hence that $H$ is extremal. Since we have assumed $H$ not to be extremal,
  this is a contradiction and $m\geq 2$ must hold.

  In this case, we have $\Delta_{q,m+1}[H] \unlhd \Delta_{q,m+1}^{*}[H]$ by
  Theorem \ref{sec:q-Laguerre-thm} and $\Delta_{q,m+1}[H] \neq_{\R}
  \Delta_{q,m+1}^{*}[H]$ since $H$ is not extremal. (\ref{eq:2}) therefore
  implies that (i) $F_{y}\in\overline{\mathcal{R}}_{m}(q)\setminus
  \mathcal{R}_{m}(q)$, (ii) $\Delta_{q,m+1}[H]$ and $\Delta_{q,m+1}^{*}[H]$ have
  a common zero $w^{*}\neq 0$, and (iii) either
  $\Delta_{q,m+1}[H]\in\overline{\mathcal{N}}_{m}(q)\setminus
  \mathcal{N}_{m}(q)$ or $F_{y}$ is extremal. For if one of the three conditions
  (i)--(iii) would not hold, then, because of Theorem \ref{sec:thm-inv-lhd-vee},
  (\ref{eq:2}) could not hold for $x^{*}\neq 0$. Since $m\geq 2$, Statement (i)
  obviously implies $F\in\overline{\mathcal{R}}_{m+1}(q)\setminus
  \mathcal{R}_{m+1}(q)$, and (ii) is equivalent to $H(w^{*}) =
  H(q^{-1}w^{*})=0$, which means $H\in\overline{\mathcal{N}}_{m+1}(q)\setminus
  \mathcal{N}_{m+1}(q)$.

  If $y=\infty$, then $\deg F \leq m$ and (\ref{eq:50}) implies
  \begin{equation*}
    F*_{q}^{m+1} H = \Delta_{q,m+1}^{*}[H] *_{q}^{m} F \quad \mbox{and} \quad
    F_{y}*_{q}^{m+1} H =  -zF*_{q}^{m+1} H= -z(\Delta_{q,m+1}[H] *_{q}^{m} F) .
  \end{equation*}
  Thus, in this case (\ref{eq:51}) implies that $\Delta_{q,m+1}[H] *_{q}^{m}
  F$ and $\Delta_{q,m+1}^{*}[H]*_{q}^{m} F$ have the common zero $x^{*}\neq 0$,
  and we can proceed as in the case $y\in\R$ to find that this can only hold if
  $F\in\overline{\mathcal{R}}_{m+1}(q)\setminus \mathcal{R}_{m+1}(q)$ and
  $H\in\overline{\mathcal{N}}_{m+1}(q)\setminus \mathcal{N}_{m+1}(q)$.

  The proof of Claim 1 is thus complete.
 
  \textit{Proof of Claim 2.}  Let
  $F\in\overline{\mathcal{N}}_{m+1}(q)\setminus\{0\}$,
  $G\in\pi_{m+1}(\R_{0}^{-})\setminus\{0\}$ be such that $F\unlhd G$ and $F
  \neq_{\R} G$ and suppose that
  $H\in\overline{\mathcal{R}}_{m+1}(q)\setminus\{0\}$. We assume first that
  $H\in\mathcal{R}_{m+1}(q)\setminus\{0\}$ and that $F$, $G$, $H$ do not
  vanish at the origin and are all of degree $m+1$.

  Note first that our assumptions and Lemma \ref{sec:charact-lhd-vee} imply
  \begin{equation}
    \label{eq:52}
    \left(\frac{F*_{q}^{m+1} H}{G*_{q}^{m+1} H}\right)(0) = 
    \left(\frac{F}{G}\right)(0) > 0.
  \end{equation}
  Next, observe that Theorems \ref{sec:q-rolle-thm}, \ref{sec:q-Laguerre-thm-1},
  and \ref{sec:q-Laguerre-thm}, show $\Delta_{q,m+1}[H]$,
  $\Delta_{q,m+1}^{*}[H]\in\mathcal{R}_{m}(q)$ and $\Delta_{q,m+1}[H]\prec
  \Delta_{q,m+1}^{*}[H]$. Theorem \ref{sec:thm-inv-prec-prec} (which holds for
  $n=m$ by assumption) thus gives
  \begin{equation*}
    \Delta_{q,m+1}[H] *_{q}^{m} F_{y} \prec  \Delta_{q,m+1}^{*}[H]*_{q}^{m} F_{y},
  \end{equation*}
  or, equivalently,
  \begin{equation*}
    \Delta_{q,m+1}^{*}[H] *_{q}^{m} F_{y} \vee
    z\left(\Delta_{q,m+1}[H]*_{q}^{m} F_{y}\right) 
  \end{equation*}
  for every zero $y$ of $F$. Since all zeros of $F$ are non-positive, this
  implies, by Lemma \ref{sec:analogue-s-t-lemma-1},
  \begin{equation*}
    \Delta_{q,m+1}^{*}[H]*_{q}^{m} F_{y} \vee   
    z\left(\Delta_{q,m+1}[H]*_{q}^{m} F_{y}\right) - 
    y\Delta_{q,m+1}^{*}[H] *_{q}^{m} F_{y}.
  \end{equation*}
  By (\ref{eq:50}) this means that 
  \begin{equation*}
    H *_{q}^{m+1} F_{y} \vee H *_{q}^{m+1} zF_{y} - y H *_{q}^{m+1} F_{y} =
    H *_{q}^{m+1} F
  \end{equation*}
  for all zeros $y$ of $F$.  Defining the linear operator $L:\R_{m+1}[z]
  \rightarrow \R_{m+1}[z]$ by $L[P]:=H *_{q}^{m+1} P$ for $P\in\R_{m+1}[z]$, we
  thus obtain $L[F_{y}] \vee L[F]$ for every zero $y$ of $F$. Because of
  (\ref{eq:52}) and Lemma \ref{sec:polynomials-with-log-3} this implies $L[F]
  \vee L[G]$, or $F *_{q}^{m+1} H \vee G *_{q}^{m+1} H$.
  
  Applying Lemma \ref{sec:an-extension-polyas-1}, it follows from this special
  case that for every $H\in\overline{\mathcal{R}}_{m+1}(q)$ and all
  $F\in\overline{\mathcal{N}}_{m+1}(q)$, $G\in\pi_{m+1}(\R_{0}^{-})$ with
  $F\unlhd G$ we have $F*_{q}^{m+1} H \veebar G*_{q}^{m+1} H$. As in the proof
  of Claim 1 one can use this result to show that $F*_{q}^{m+1} H \vee
  G*_{q}^{m+1} H$ if $F\lhd G$.

  In order to prove that, for non-extremal
  $H\in\overline{\mathcal{R}}_{m+1}(q)$, we have
  \begin{equation}
    \label{eq:5}
    F*_{q}^{m+1} H \vee G*_{q}^{m+1} H \quad\mbox{if}\quad
    F\in\mathcal{N}_{m+1}(q) \quad\mbox{or}\quad H\in\mathcal{R}_{m+1}(q),
  \end{equation}
  suppose $F\in\overline{\mathcal{N}}_{m+1}(q)$ and $G\in\pi_{m+1}(\R_{0}^{-})$
  satisfy $F\unlhd G$ and $F\neq_{\R} G$, and suppose
  $H\in\overline{\mathcal{R}}_{m+1}(q)$, not extremal, is such that
  $F*_{q}^{m+1} H$ and $G*_{q}^{m+1} H$ have a common zero at a point $x^{*}<0$
  (the case in which the common zero is positive can be treated in a similar
  way).  Denote the sets of $m+1$-zeros of $F$ and $G$ by, respectively,
  $\mathcal{Z}_{F}$ and $\mathcal{Z}_{G}$.

  Since $F\in\overline{\mathcal{N}}_{m+1}(q)$ implies $F_{y} \unlhd F$ for every
  zero $y$ of $F$ and $F \unlhd -F_{\infty}$ if $\deg F \leq m$, our results so
  far show $L[F_{y}] \veebar L[F]$ and $L[F] \veebar -L[F_{\infty}]$. Therefore,
  for every $y\in\mathcal{Z}_{F}$ we either have $L[F_{y}](x^{*})=0$ or, because
  of Lemma \ref{sec:polynomials-with-log-2},
  \begin{equation}
    \label{eq:54}
    \left(\frac{F*_{q}^{n+1} H}{F_{y}*_{q}^{n+1} H}\right)'(x^{*}) > 0.
  \end{equation}
  If (\ref{eq:54}) would actually hold for one
  $y\in\mathcal{Z}_{F}\setminus\mathcal{Z}_{G}$, then Lemma \ref{sec:working}
  would imply that $(G*_{q}^{m} H)(x^{*})\neq 0$. Consequently, there must be
  at least one $y\in\mathcal{Z}_{F}$ with
  \begin{equation*}
    (F_{y}*_{q}^{m+1} H)(x^{*})= 0 = (F*_{q}^{m+1} H)(x^{*})
  \end{equation*}
  and from that point on one can argue as in the proof of Claim 1 in order to
  show that (\ref{eq:5}) must be true.

  The proof of Claim 2, and thus also of Theorems \ref{sec:thm-inv-lhd-vee} and
  \ref{sec:thm-inv-prec-prec}, is therefore complete.
\end{proof}

The next result is essentially equivalent to (\ref{item:26})--(\ref{item:28}) of
Theorem \ref{sec:real-zeros-main-thm}.

\begin{theorem}
  \label{sec:main-conv-thm}
  Suppose $0<r<q\leq 1$ and $F\in\R_{n}[z]$.
  \begin{enumerate}
  \item \label{item:50} We have $F*_{q}^{n} G \in \mathcal{R}_{n}(q)$ for all
    $G\in \mathcal{R}_{n}(q)$ if, and only if,
    $F\in\overline{\mathcal{N}}_{n}(q)$.
  \item \label{item:51} We have $F*_{q}^{n} G \in \mathcal{N}_{n}(q)$ for all
    $G\in \mathcal{N}_{n}(q)$ if, and only if,
    $F\in\overline{\mathcal{N}}_{n}(q)$.
  \item \label{item:52} If $F$ belongs to $\overline{\mathcal{R}}_{n}(q)$ or
    $\overline{\mathcal{N}}_{n}(q)$ and is not extremal, then $F*_{q}^{n}
    R_{n}(r;z)$ belongs to, respectively, $\mathcal{R}_{n}(r)$ or
    $\mathcal{N}_{n}(r)$. 
  \end{enumerate}
\end{theorem}

\begin{proof}
  Let $G\in\mathcal{R}_{n}(q)$ and $F\in\overline{\mathcal{N}}_{n}(q)$. Then
  $F(z) \unlhd F(q^{-1} z)$ by Lemma \ref{sec:working-2} and therefore
  Theorem \ref{sec:thm-inv-lhd-vee} shows that 
  \begin{equation*}
    (F*_{q}^{n}G)(z) = F(z)*_{q}^{n}G(z) \vee F(q^{-1} z)*_{q}^{n}G(z) =
    (F*_{q}^{n}G)(q^{-1}z)  
  \end{equation*}
  since $G\in\mathcal{R}_{n}(q)$. Because of Lemma \ref{sec:working-2} this
  is equivalent to $F*_{q}^{n}G\in\mathcal{R}_{n}(q)$. 

  If $F*_{q}^{n} G \in \mathcal{R}_{n}(q)$ for all $G\in \mathcal{R}_{n}(q)$,
  then the choice $G(z)=R_{n}(q;z)$ shows that $F \in
  \overline{\mathcal{R}_{n}(q)}$. In order to show that in fact either $F(z)$ or
  $F(-z)$ must belong to $\overline{\mathcal{N}}_{n}(q)$ one can argue as in the
  proof of \cite[Thms. 1.I, 3.I]{polschur14}. One simply has to consider the
  polynomials $x^{\nu-1} - q x^{\nu+1}$ and $x^{\nu-1} +(q+1)x^{\nu}+ q
  x^{\nu+1}$ instead of, respectively, the two polynomials $x^{\nu-1} -
  x^{\nu+1}$ and $x^{\nu-1} + 2x^{\nu}+ x^{\nu+1}$ which appear in the formula
  before equation (5) in \cite{polschur14}, and to use Theorem
  \ref{sec:q-newton-ineq} instead of the classical ''Newton's inequalities''.
  (\ref{item:50}) is therefore proven.
 
  If $F\in\overline{\mathcal{N}}_{n}(q)$, $G\in \mathcal{N}_{n}(q)$, then we can
  assume that all coefficients of $F$ and $G$, and therefore also of
  $F*_{q}^{n}G$, are non-negative. This means that $F*_{q}^{n}G$ cannot vanish
  for positive $z$. Since from (\ref{item:50}) we know that $F*_{q}^{n} G \in
  \mathcal{R}_{n}(q)$, this shows $F*_{q}^{n} G \in
  \mathcal{N}_{n}(q)$. Moreover, $F*_{q}^{n} G \in \mathcal{N}_{n}(q)$ for all
  $G\in \mathcal{N}_{n}(q)$ clearly implies $F = F*_{q}^{n} R_{n}(q;z) \in
  \overline{\mathcal{N}}_{n}(q)$. This proves (\ref{item:51}).

  If $r\in(0,q)$, then $R_{n}(r;z) \in\mathcal{N}_{n}(q)$ and $R_{n}(r;z)
  \unlhd R_{n}(r;r^{-1} z)$. Theorem \ref{sec:thm-inv-lhd-vee} thus yields
  \begin{equation*}
    (F*_{q}^{n}R_{n}(r;z))(z) = F*_{q}^{n}R_{n}(r;z) \vee
    F*_{q}^{n}R_{n}(r;r^{-1}z) =   (F*_{q}^{n}R_{n}(r;z))(r^{-1}z)
  \end{equation*}
  for every $F\in\overline{\mathcal{R}}_{n}(q)$ that is not extremal. Because
  of Lemma \ref{sec:working-2} this is equivalent to
  $F*_{q}^{n}R_{n}(r;z)\in\mathcal{R}_{n}(r)$. If $F$ belongs to
  $\overline{\mathcal{N}}_{n}(q)$ and is not extremal, then
  $F*_{q}^{n}R_{n}(r;z)\in\mathcal{R}_{n}(r)$ and the coefficients of
  $F*_{q}^{n}R_{n}(r;z)$ are either all non-positive or all
  non-negative. This implies $F*_{q}^{n}R_{n}(r;z)\in\mathcal{N}_{n}(r)$ and thus
  completes the proof of (\ref{item:52}).
\end{proof}

In order to complete the proof of Theorem \ref{sec:real-zeros-main-thm}, we
still need the following converse of ''Newton's Inequalities''.

\begin{lemma}
  \label{sec:q-newton-conv}
  Let $f\in\R_{n}[z]$.
  \begin{enumerate}
  \item \label{item:53} If $f\in \mathcal{LC}_{n}$ then there is a
    $q\in(0,1]$ such that $f*R_{n}(q;z)\in\mathcal{R}_{n}(q)$.
  \item \label{item:54} If $f\in \mathcal{LC}_{n}^{+}$ then there is a
    $q\in(0,1]$ such that $f*R_{n}(q;z)\in\mathcal{N}_{n}(q)$.
  \end{enumerate}
\end{lemma}

\begin{proof}
  Since $f\in \mathcal{LC}_{n}^{+}$ implies that all coefficients of
  $f*R_{n}(q;z)$ are either $0$ or of same sign, (\ref{item:54}) follows
  directly from (\ref{item:53}).

  In order to prove (\ref{item:53}) we will assume that $f(z)=\sum_{k=0}^{n}
  a_{k} z^{k}\in\mathcal{LC}_{n}$ with $f(0)\neq 0$ and $\deg f = n$ (the
  general case being only slightly more difficult technically). Hence
  $a_{k}^{2}>a_{k-1}a_{k+1}$ for all $k\in\{0,\ldots,n\}$. In particular, we
  must have $a_{k-1}a_{k+1} < 0$ if $a_{k}=0$.

  Set $F_{q}(z):= f(z)*R_{n}(q;z)$ and observe that, for $k$, $m\in\N_{0}$,
  \begin{equation*}
    m(m+1)+k(k-1)-2mk = (k-m)(k-(m+1)) 
  \end{equation*}
  and
  \begin{equation*}
    m^{2}-1 +
    k(k-1)-2mk+k = (k-(m-1))(k-(m+1)).
  \end{equation*}
  It therefore follows from (\ref{eq:33}) that for
  $m\in\{0,\ldots,n-1\}$ and $q\rightarrow 0$
  \begin{equation*}
    q^{(m+1)m/2}F_{q}(q^{-m}z) = \sum_{k=0}^{n} q^{(m+1)m/2-mk} C_{k}^{n}(q)
    a_{k} z^{k} \rightarrow a_{m}z^{m} +a_{m+1} z^{m+1}, 
  \end{equation*}
  and, if $a_{m}=0$ for an $m\in\{1,\ldots,n-1\}$,
  \begin{equation*}
    q^{(m^{2}-1)/2}F_{q}(q^{-m+1/2}z) = \sum_{k=0}^{n} q^{(m^{2}-1)/2-mk+k/2} C_{k}^{n}(q)
    a_{k} z^{k} \rightarrow a_{m-1}z^{m-1} +a_{m+1} z^{m+1}. 
  \end{equation*}
  Consequently, for every $m\in\{0,\ldots,n-1\}$ for which $a_{m}a_{m+1}\neq
  0$ there is a zero $z_{m}(q)$ of $F_{q}$ with 
  \begin{equation}
    \label{eq:56}
    z_{m}(q) \sim -q^{-m}a_{m}/a_{m+1}\quad \mbox{as} \quad q\rightarrow 0.
  \end{equation}
  If $a_{m}=0$ for an $m\in\{1,\ldots,n-1\}$, then there are zeros
  $z_{m-1}(q)$ and $z_{m}(q)$ of $F(q)$ with 
  \begin{equation}
    \label{eq:57}
    z_{m-1}(q) \sim
    -q^{-m+1/2}\sqrt{-a_{m-1}/a_{m+1}}\quad \mbox{and}\quad z_{m}(q) \sim
    q^{-m+1/2}\sqrt{-a_{m-1}/a_{m+1}} 
  \end{equation}
  as $q\rightarrow 0$.  Since $F_{q}$ is a real polynomial, this shows that
  for all $q>0$ sufficiently close to $0$ we must have $F_{q}\in\pi_{n}(\R)$
  and
  \begin{equation*}
    |z_{0}(q)|\leq |z_{1}(q)| \leq \cdots \leq |z_{n-1}(q)|\leq |z_{n}(q)| 
  \end{equation*}
  with $|z_{m}(q)| = |z_{m+1}(q)|$ if, and only if, $a_{m+1}=0$.

  Now, if $z_{m}(q)$ and $z_{m+1}(q)$ are of same sign and $a_{m}a_{m+1} \neq
  0$, then
  \begin{equation*}
    \frac{z_{m}(q)}{z_{m+1}(q)} \sim q \frac{a_{m}a_{m+2}}{a_{m+1}^{2}}
    < q 
  \end{equation*}
  for all $q>0$ close to $0$, because of (\ref{eq:56}) and since $a_{m}a_{m+2}
  < a_{m+1}^{2}$. 

  If $l$, $m\in\{0,\ldots,n-1\}$ with $l<m-1$ are such that $z_{l}(q)$ and
  $z_{m}(q)$ are of same sign and $a_{l}\neq 0\neq a_{m}$, (\ref{eq:56}) shows
  that
  \begin{equation*}
    \frac{z_{l}(q)}{z_{m}(q)} \sim q^{m-l} \frac{a_{l}a_{m+1}}{a_{l+1}a_{m}}
    < q^{2} \frac{a_{l}a_{m+1}}{a_{l+1}a_{m}} < q
  \end{equation*}
  for all $q>0$ close to $0$.  

  If $l$, $m\in\{0,\ldots,n-1\}$ with $l\leq m-1$ are such that $z_{l}(q)$ and
  $z_{m}(q)$ are of same sign and $a_{l}= 0$, $a_{m}\neq 0$, then, in the case
  $l=m-1$,
  \begin{equation*}
    \frac{z_{l}(q)}{z_{m}(q)} \sim q^{m-l+ 1/2}
    \left|\frac{a_{m+1}}{a_{m}} \sqrt{-\frac{a_{l-1}}{a_{l+1}}}\right|
    < q^{3/2}  \left|\frac{a_{m+1}}{a_{m}}
      \sqrt{-\frac{a_{l-1}}{a_{l+1}}}\right| < q
  \end{equation*}
  for all $q>0$ close to $0$, whereas in the case $l<m-1$
  \begin{equation*}
    \frac{z_{l}(q)}{z_{m}(q)} \sim q^{m-l\pm 1/2}
    \left|\frac{a_{m+1}}{a_{m}} \sqrt{-\frac{a_{l-1}}{a_{l+1}}}\right|
    < q^{3/2}  \left|\frac{a_{m+1}}{a_{m}}
      \sqrt{-\frac{a_{l-1}}{a_{l+1}}}\right| < q
  \end{equation*}
  for all $q>0$ close to $0$. In the same way one verifies that also in the
  remaining two cases $a_{l}\neq 0$, $a_{m}= 0$, and $a_{l}= 0$, $a_{m}= 0$,
  one has $z_{l}/z_{m}<q$ for all $q$ close to $0$. 

  This shows that all zeros of $F_{q}$ of equal sign are strictly
  $q$-separated when $q>0$ is close to $0$, and hence that
  $F_{q}\in\mathcal{R}_{n}(q)$ for those $q$.
\end{proof}

\begin{proof}[Proof of Theorem \ref{sec:real-zeros-main-thm}]
  For $0<r<q\leq 1$ Statements (\ref{item:26})--(\ref{item:28}) follow readily
  from Theorem \ref{sec:main-conv-thm}. Moreover, (\ref{item:26}) and
  (\ref{item:27}) are trivial for $q=0$. We have already shown that every
  $f\in\overline{\mathcal{PR}}_{n}(q)$ or $f\in\overline{\mathcal{PN}}_{n}(q)$
  belongs to, respectively, $\mathcal{PR}_{n}(r)$ or $\mathcal{PN}_{n}(r)$ if
  $0<r<q\leq 1$. Since $\mathcal{PN}_{n}(r)\subset \mathcal{PN}_{n}(0)$ by
  definition, we have thus verified (\ref{item:26})--(\ref{item:28}).

  If $F(z) = \sum_{k=0}^{n} C_{k}^{n}(q) a_{k} z^{k}
  \in\overline{\mathcal{R}}_{n}(q)$ is not extremal, then it follows from
  Theorem \ref{sec:main-conv-thm}(\ref{item:52}) that $F*_{q}^{n}R_{n}(r;z) =
  \sum_{k=0}^{n} C_{k}^{n}(r) a_{k} z^{k}\in \mathcal{R}_{n}(r)$ for every
  $r\in(0,q)$. Theorem \ref{sec:q-newton-ineq} thus implies $f(z)=
  \sum_{k=0}^{n} a_{k} z^{k}\in \mathcal{LC}_{n}$. The other inclusion of
  Theorem \ref{sec:real-zeros-main-thm}(\ref{item:29}) is verified in Lemma
  \ref{sec:q-newton-conv}. Hence, Theorem
  \ref{sec:real-zeros-main-thm}(\ref{item:29}) is verified for the classes
  $\mathcal{R}_{n}(0)$. The proof for the classes $\mathcal{N}_{n}(0)$ is very
  similar and therefore the proof of Theorem \ref{sec:real-zeros-main-thm} is
  complete.
\end{proof}

Finally, we will show how Theorem \ref{sec:thm-inv-prec-prec} can be used to
obtain a $q$-extension of Corollary \ref{sec:introduction-1}(\ref{item:6}). 

If we denote the open upper half-plane by $\mathbb{U}$, then the Hermite-Biehler
theorem \cite[Thm. 6.3.4]{rahman} states that
\begin{equation*}
  \pi_{n}(\mathbb{U}) = \left\{F+iG: F,G\in \pi_{n}(\R) 
    \mbox{ and } F\prec G\right\}.
\end{equation*}
Consequently, if for $q\in(0,1]$ we define
\begin{equation*}
  \mathcal{U}_{n}(q) := \left\{F+iG: F,G\in \overline{\mathcal{R}}_{n}(q) 
    \mbox{ and } F\prec G\right\},
\end{equation*}
then $\mathcal{U}_{n}(1) = \pi_{n}(\mathbb{U})$ and the following, easily
verified consequence of Theorem \ref{sec:thm-inv-prec-prec} is the desired
$q$-extension of Corollary \ref{sec:introduction-1}(\ref{item:6}).

\begin{theorem}
  \label{sec:q-halfplane-ext}
  Let $q\in(0,1]$. Then $\mathcal{M}(\mathcal{U}_{n}(q)) = \{ f\in
  \overline{\mathcal{PN}}_{n}(q): f(0)\neq 0\}$.
\end{theorem}

\section{An extension of Ruscheweyh's convolution lemma}
\label{sec:proofs-theorems}

In this section we will prove the extension of Ruscheweyh's convolution lemma
that is given by Lemma \ref{sec:ext-rusch-lemma}. We will obtain Lemma
\ref{sec:ext-rusch-lemma} as a limit case of a version of Lemma
\ref{sec:main-results-2} in which polynomials which are symmetric with respect
to $\R$ (i.e. real polynomials) are replaced by polynomials which are symmetric
with respect to $\T$ (so-called self-inversive polynomials). Lemma
\ref{sec:main-results-2} can therefore be seen as the real polynomial version of
Ruscheweyh's convolution lemma. The necessary definitions regarding
self-inversive polynomials are as follows.

The \emph{$n$-inverse} of a polynomial $F(z)=\sum_{k=0}^{n}a_{k} z^{k}$ of
degree $\leq n$ is defined by 
\begin{equation*}
  I_{n}[F](z) := z^{n} \overline{F\left(\frac{1}{\overline{z}}\right)} = 
  \sum_{k=0}^{n}\overline{a}_{n-k} z^{k} 
\end{equation*}
and $F$ is called \emph{$n$-self-inversive} if $F = I_{n}[F]$ (in particular $0$
is $n$-self-inversive for all $n\in\N_{0}$). The zeros of $I_{n}[F]$ are
obtained by reflecting the zeros of $F$ with respect to $\T$. Hence, if
$F\in\pi_{n}(\D)$, then $F/I_{n}[F]$ is a Blaschke product, and therefore, for
those $F$, we have $F + \zeta I_{n}[F] \in \pi_{n}(\T)$ for all $\zeta\in\T$. The
zero reflection property of $I_{n}[F]$ also shows that the zeros of
$n$-self-inversive polynomials lie symmetrically around $\T$. Furthermore, it is
easy to see that every polynomial of degree $\leq n$ with zeros symmetrically
around $\T$ is $n$-self-inversive up to a constant multiple of modulus $1$.

It is clear that $F(z)=\sum_{k=0}^{n}a_{k} z^{k}$ is $n$-self-inversive if, and
only if, $a_{k} = \overline{a}_{n-k}$ for all $k\in\{0,\ldots,n\}$ and therefore
$\mathcal{SI}_{n}$, the set of all $n$-self-inversive polynomials, is a real
vector space of dimension $n+1$. The coefficient symmetry of $n$-self-inversive
polynomials also implies that for $F\in\mathcal{SI}_{n}$ we have $e^{-int/2}
F(e^{it})\in \R$ for all $t\in\R$.

\begin{lemma}
  \label{sec:blaschke-lemma}
  \begin{enumerate}
  \item[]
  \item\label{item:55} For all $F\in\C_{n}[z]$ and $m\in\N_{0}$ we have $F+ z^{m}
    I_{n}[F]\in\mathcal{SI}_{n+m}$.
  \item\label{item:56} For all $F\in\pi_{n}(\C\setminus\D)$ and $m\in\N_{0}$ we
    have $F+ z^{m} I_{n}[F]\in\pi_{n+m}(\T)$.
  \end{enumerate}
\end{lemma}

\begin{proof}
  We have
  \begin{equation*}
    I_{n+m}[F+ z^{m} I_{n}[F]] = I_{n+m}[F]+ I_{n+m}[z^{m} I_{n}[F]] 
    = z^{m} I_{n}[F] + F
  \end{equation*}
  and thus (\ref{item:55}) is clear. Since the zeros of $I_{n}[F]$ are obtained
  by reflecting the zeros of $F$ around $\T$, $I_{n}[F]/F$, and thus also
  $z^{m}I_{n}[F]/F$, is a Blaschke product when
  $F\in\pi_{n}(\C\setminus\D)$. Since a Blaschke product can take the value $1$
  only on $\T$, (\ref{item:55}) is also proven.
\end{proof}

We say that $F$, $G\in\pi_{n}(\T)$ have \emph{$\T$-interspersed} zeros if the
zeros of $F$ and $G$ alternate on the unit circle. If $F$ and $G$ have
$\T$-interspersed zeros, but no common zeros, then $F$ and $G$ are said to have
\emph{strictly $\T$-interspersed zeros}. The following analogue of Lemma
\ref{sec:polyn-with-intersp-1} holds for $\T$-interspersion: $F\in \pi_{n}(\T)$
and $G\in\mathcal{SI}_{n}\setminus\{0\}$ have $\T$-interspersed zeros if, and
only if, the real valued function
\begin{equation*}
  t\mapsto \frac{F(e^{it})}{G(e^{it})} = 
  \frac{e^{-int/2}F(e^{it})}{e^{-int/2}G(e^{it})}, \qquad t\in\R,
\end{equation*}
is either strictly increasing on $\R$ or strictly decreasing on $\R$. Similarly
to the real case we therefore write $F \preceq_{\T} G$ if $F$, $G\in\pi_{n}(\T)$
satisfy
\begin{equation*}
  (e^{-int/2}F(e^{it}))'(e^{-int/2}G(e^{it})) -
  (e^{-int/2}F(e^{it}))(e^{-int/2}G(e^{it}))'\leq 0\quad \mbox{for} \quad t\in\R,
\end{equation*}
and $F\prec_{\T} G$ if $F\preceq_{\T} G$ and $F$ and $G$ do not have common
zeros. It is then easy to see that the following holds.

\begin{lemma}
  \label{sec:proof-mathc-vers-1}
  Let $F$, $G\in\mathcal{SI}_{n}$. Then $F\preceq_{\T} G$ if, and only if,
  $\Im (F/G)(z) < 0$ for $z\in\D$.
\end{lemma}

Using the M\"obius transformation $i(1+z)/(1-z)$ we can transfer Lemma
\ref{sec:main-results-2} to $n$-self-inversive polynomials as follows.

\begin{lemma}
  \label{sec:main-results-3}
  Let $L:\mathcal{SI}_{n}\rightarrow \mathcal{SI}_{m}$ be a real linear operator
  and suppose $F$, $G\in\mathcal{SI}_{n}$ are such that $F/G = P/Q$ with
  polynomials $P$, $Q$ that have zeros only on $\T$ and satisfy $P \prec_{\T}
  Q$. Let $\mathcal{Z}$ denote the set of $y\in\{e^{it}:0\leq t<\pi\}$ for which
  $-y^{2}$ is a zero of $F/G$. If for every $y\in\mathcal{Z}$
  \begin{equation}
    \label{eq:13}
    L\left[\frac{(1+z)F}{y + \overline{y}z}\right]\preceq_{\T}
    L\left[\frac{i(1-z)F}{y + \overline{y}z}\right] 
  \end{equation}
  then $L[F]\preceq_{\T} L[G]$. If (\ref{eq:13}) holds with $\preceq_{\T}$
  replaced by $\prec_{\T}$ for one $y\in\mathcal{Z}$, then $L[F]\prec_{\T}
  L[G]$.
\end{lemma}

\begin{proof}
  Set
  \begin{equation*}
    \psi(z) := i\frac{1+z}{1-z} \quad \mbox{and note that}\quad
    \psi^{(-1)}(z) = \frac{z-i}{z+i}.
  \end{equation*}
  Then
  \begin{equation*}
    \Psi_{n}[H](z):=(z+i)^{n} H(\psi^{(-1)}(z)), \qquad H\in\C_{n}[z],
  \end{equation*}
  with inverse
  \begin{equation*}
    \Psi_{n}^{(-1)}[H](z):=\frac{1}{(2i)^{n}}(1-z)^{n} H(\psi(z))
  \end{equation*}
  is an isomorphism between $\mathcal{SI}_{n}$ and $\R_{n}[z]$ which maps
  $\pi_{n}(\T)$ onto $\pi_{n}(\R)$ and $\sigma_{n}(\T)$ onto
  $\sigma_{n}(\R)$. Moreover, since $(\psi(e^{it}))' >0$ for $t\in(0,2\pi)$,
  $\Psi_{n}$ \emph{preserves position}, i.e. we have $F \preceq_{\T} G$ and
  $F\prec_{\T} G$ if, and only if, $\Psi_{n}[F] \preceq \Psi_{n}[G]$ and
  $\Psi_{n}[F]\prec \Psi_{n}[G]$, respectively.

  Straightforward calculations show that if $y\in\mathcal{Z}$, then
  \begin{equation}
    \label{eq:58}
    \Psi_{n}\left[\frac{(1+z)F}{y+\overline{y}z}\right] =
    \frac{2}{y+\overline{y}}\frac{z\Psi_{n}[F]}{z-\psi(-y^{2})}
    \quad \mbox{and}\quad  
    \Psi_{n}\left[\frac{i(1-z)F}{y+\overline{y}z}\right] =
    \frac{-2}{y+\overline{y}}\frac{\Psi_{n}[F]}{z-\psi(-y^{2})} 
  \end{equation}
  if $y\neq i$, and
  \begin{equation}
    \label{eq:55}
    \Psi_{n}\left[\frac{(1+z)F}{y+\overline{y}z}\right] =
    -z\Psi_{n}[F]
    \quad \mbox{and}\quad  
    \Psi_{n}\left[\frac{i(1-z)F}{y+\overline{y}z}\right] =
    \Psi_{n}[F]
  \end{equation}
  if $y = i$. Note that $\Psi_{n}[F]$ is of degree $n$ if, and only if,
  $i\notin\mathcal{Z}$ and that $\{\psi(-y^{2}):
  y\in\mathcal{Z}\setminus\{0\}\}$ is the set of zeros of $\Psi_{n}[F]$.

  Hence, if we set $A:=\Psi_{n}[F]$ and define
  \begin{equation*}
    K: \R_{n}[z] \rightarrow \R_{m}[z],\, H\mapsto (\Psi_{m} \circ L \circ
    \Psi_{n}^{(-1)})[H], 
  \end{equation*}
  then it follows from (\ref{eq:13}), (\ref{eq:55}), and the fact that
  $\Psi_{n}$ preserves position, that
  \begin{equation}
    \label{eq:59}
    K[A_{\infty}] \preceq K[A] \quad \mbox{if} \quad \deg A < n,
  \end{equation}
  and from (\ref{eq:13}) and (\ref{eq:58}) that $K[zA_{x}] \preceq -K[A_{x}]$
  for every zero $x$ of $A$. Because of Lemma \ref{sec:lemmas-2} this implies
  $K[A]=K[zA_{x}] -x K[A_{x}] \preceq -K[A_{x}]$, and thus we obtain
  \begin{equation}
    \label{eq:60}
    K[A_{x}] \preceq K[A] \quad \mbox{for every zero} \quad x 
    \quad \mbox{of} \quad A. 
  \end{equation}
  
  Since $\Psi_{n}$ preserves position, $F\preceq_{\T} G$ implies $A\preceq B$,
  with $B:=\Psi_{n}[G]$. It therefore follows from (\ref{eq:59}), (\ref{eq:60}),
  and Lemma \ref{sec:main-results-2}, that $K[A] \preceq K[B]$. This implies
  $L[F]\preceq_{\T} L[G]$ and the proof is complete.
\end{proof}

\begin{lemma}
  \label{sec:lemmas-4}
  Let $F$ and $G$ be polynomials of degree $\leq n$ that are such that 
  \begin{equation*}
    \Im \frac{F(z)}{G(z)} < 0 \quad\mbox{for}\quad z\in\D. 
  \end{equation*}
  Then 
  \begin{equation*}
    \Im \frac{F(z) + z^{n+1} I_{n}[F](z)}{G(z) + z^{n+1} I_{n}[G](z)} < 0 \quad
    \mbox{for}\quad z\in \D.
  \end{equation*}
\end{lemma}

\begin{proof}
  It follows from the assumptions that $(F - xG)(z) \neq 0$ for all $x\in\R$ and
  $z\in\D$. By Lemma \ref{sec:blaschke-lemma}(\ref{item:56}) this implies
  \begin{equation*}
    F(z)-xG(z) + z^{n+1} I_{n}[F-xG](z)  \neq 0, 
  \end{equation*}
  or, equivalently, since $I_{n}$ is real linear,
  \begin{equation*}
    \frac{F(z) + z^{n+1}I_{n}[F](z)}{G(z)+z^{n+1} I_{n}[G](z)} \neq x
  \end{equation*}
  for all $x\in\R$ and $z\in\D$. The assertion thus follows from the fact that
  $\Im (F/G)(0)<0$.
\end{proof}

\begin{lemma}
  \label{sec:main-results-4}
  Let $L:\mathcal{H}(\D)\rightarrow\mathcal{H}(\D)$ be a continuous real linear
  operator. Suppose $f\in\mathcal{H}(\D)$ is such that
  \begin{equation}
    \label{eq:14}
    \Im \frac{L\left[\frac{(1+z)f}{y + \overline{y}z}\right]}
    {L\left[\frac{i(1-z)f}{y + \overline{y}z}\right]} (z) < 0 \quad \mbox{for
      all} \quad z\in\D \mbox{ and } y\in\T\mbox{ with } \Im y\geq 0.
  \end{equation}
  Then for every $g\in \mathcal{H}(\D)$ which satisfies $\Im (f/g)(z) < 0$ for
  $z\in\D$ we have
  \begin{equation*}
    \Im \frac{L\left[f\right]}{L\left[g\right]} (z) < 0 \quad \mbox{for
      } \quad z\in\D.
  \end{equation*}
\end{lemma}

\begin{proof}
  By considering $h\mapsto L[h](rz)$, $r\in(0,1)$, instead of $L$, and
  $f(s_{r}z)$ and $g(s_{r}z)$ instead of $f$ and $g$ for a suitable function
  $s_{r}\in(0,1)$ with $\lim_{r\rightarrow 1} s_{r} =1$, we can assume that $\Im
  (f/g)(z) < 0$ for $z\in\overline{\D}$ and that (\ref{eq:14}) holds for
  $z\in\overline{\D}$.

  Now, let
  \begin{equation*}
    \Phi_{n}:\mathcal{H}(\D)\rightarrow \mathcal{H}(\D), \sum_{k=0}^{\infty} a_{k}
    z^{k} \mapsto \sum_{k=0}^{n} a_{k} z^{k}
  \end{equation*}
  and set
  \begin{equation*}
    L_{n}[h]:=(\Phi_{n}\circ L\circ \Phi_{n})[h] \quad \mbox{for} \quad
    h\in\mathcal{H}(\D),\, n\in\N.
  \end{equation*}
  Then $\{L_{n}\}_{n}$ is a pointwise convergent sequence of continuous linear
  operators and thus an equicontinuous family.

  Setting $h_{n}:=\Phi_{n}[h]$ for $h\in\mathcal{H}(\D)$, it therefore follows
  from (\ref{eq:14}) and a compactness argument that there is an $n_{0}\in\N$
  such that
  \begin{equation}
    \label{eq:15}
    \Im \frac{L_{n}\left[\frac{(1+z)f_{n}}{y + \overline{y}z}\right]}
    {L_{n}\left[\frac{i(1-z)f_{n}}{y + \overline{y}z}\right]} (z) < 0 \quad
    \mbox{for all} \quad z\in\overline{\D},\,y\in\T':=\{z\in\T: \Im
    y\geq 0\},\, n\geq n_{0}. 
  \end{equation}
  By Lemma \ref{sec:lemmas-4} this means that
  \begin{equation}
    \label{eq:16}
    \Im \frac{L_{n}\left[\frac{(1+z)f_{n}}{y + \overline{y}z}\right] + z^{n+1}
      \left(L_{n}\left[\frac{(1+z)f_{n}}{y +
            \overline{y}z}\right]\right)^{*n}} 
    {L_{n}\left[\frac{i(1-z)f_{n}}{y + \overline{y}z}\right] + z^{n+1}
      \left(L_{n}\left[\frac{i(1-z)f_{n}}{y +
            \overline{y}z}\right]\right)^{*n}} < 0  
  \end{equation}
  for all $z\in\D$, $n\geq n_{0}$, $y\in\T'$.

  For $h\in\mathcal{H}(\D)$ we define now
  \begin{equation*}
    K_{n}[h]:= L_{n}[h] + z^{n+1} (L_{n}[h])^{*n}.
  \end{equation*}
  Then, because of Lemma \ref{sec:blaschke-lemma}(\ref{item:55}), $K_{n}$ is a
  real linear operator mapping $\mathcal{SI}_{2n+1}$ into itself, and we have
  \begin{equation*}
    K_{n}\left[\frac{(1+z)(f_{n}+z^{n+1}f_{n}^{*n})}{y + \overline{y}z}\right]
    = L_{n}\left[\frac{(1+z)f_{n}}{y + \overline{y}z}\right] + z^{n+1}
      \left(L_{n}\left[\frac{(1+z)f_{n}}{y +
            \overline{y}z}\right]\right)^{*n}
  \end{equation*}
  and
  \begin{equation*}
    K_{n}\left[\frac{i(1-z)(f_{n}+z^{n+1}f_{n}^{*n})}{y + \overline{y}z}\right]
    = L_{n}\left[\frac{i(1-z)f_{n}}{y + \overline{y}z}\right] + z^{n+1}
      \left(L_{n}\left[\frac{i(1-z)f_{n}}{y +
            \overline{y}z}\right]\right)^{*n}.
  \end{equation*}
  Hence, it follows from (\ref{eq:16}) that 
  \begin{equation}
    \label{eq:17}
    \Im \frac{K_{n}\left[\frac{(1+z)(f_{n}+z^{n+1}f_{n}^{*n})}{y +
          \overline{y}z}\right]}
    {K_{n}\left[\frac{i(1-z)(f_{n}+z^{n+1}f_{n}^{*n})}{y +
          \overline{y}z}\right]} (z) < 0, \quad z\in\D
  \end{equation}
  for every $y\in\T$ for which $-y^{2}$ is a zero of $f_{n} + z^{n+1}
  f_{n}^{*n}$. 

  Since $\Im (f/g)(z) < 0$ for $z\in\overline{\D}$ we can choose $n_{0}$ in
  such a way that also $\Im (f_{n}/g_{n})(z) < 0$ for $z\in\overline{\D}$ and
  $n\geq n_{0}$. It then follows from Lemma \ref{sec:lemmas-4} that
  \begin{equation*}
    \Im \frac{f_{n}(z)+z^{n+1} f_{n}^{*n}(z)}{g_{n}(z) + z^{n+1}
      g_{n}^{*n}(z)}< 0
  \end{equation*}
  for $z\in\D$ and $n\geq n_{0}$ which, by Lemma \ref{sec:proof-mathc-vers-1},
  is equivalent to
  \begin{equation*}
    f_{n}+z^{n+1} f_{n}^{*n}\preceq_{\T} g_{n} + z^{n+1} g_{n}^{*n}.
  \end{equation*}
  Consequently, Lemma \ref{sec:main-results-3} and (\ref{eq:17}) yield
  \begin{equation*}
    K_{n}[f_{n}+z^{n+1}f_{n}^{*n}] \preceq_{\T} K_{n}[g_{n}+z^{n+1}g_{n}^{*n}].
  \end{equation*}
  Because of Lemma \ref{sec:proof-mathc-vers-1} this is equivalent to
  \begin{equation*}
    \Im \frac{K_{n}[f_{n}+z^{n+1} f_{n}^{*n}]}{K_{n}[g_{n} + z^{n+1}
      g_{n}^{*n}]}(z) < 0
  \end{equation*}
  for $z\in\D$ and $n\geq n_{0}$. It is easy to see that, for every
  $h\in\mathcal{H}(\D)$, $K_{n}[h_{n}+z^{n+1} h_{n}^{*n}]$ tends to $L[h]$
  uniformly on compact subsets of $\D$ as $n\rightarrow\infty$, and therefore we
  obtain the assertion.
\end{proof}

\begin{proof}[Proof of Lemma \ref{sec:ext-rusch-lemma}]
  Writing $x = (1+it)/(1-it)$ with $t\in\R$ we see that (\ref{eq:62}) is
  equivalent to
  \begin{equation}
    \label{eq:20}
    \frac{L\left[\frac{1+z}{y+\overline{y}z}f\right]}
    {L\left[\frac{i(1-z)}{y+\overline{y}z}f\right]}(z) \neq t \quad \mbox{for
      all}\quad t\in\R,\, z\in\D, y\in\T.
  \end{equation}
  Set $A_{y}:= L\left[\frac{zf}{y+\overline{y}z}\right](0)$ and $B_{y}:=
  L\left[\frac{f}{y+\overline{y}z}\right](0)$. The assertion then follows from
  (\ref{eq:20}) and Lemma \ref{sec:main-results-4}, since by (\ref{eq:63})
  there is a $y\in\T$ such that
  \begin{equation*}
    \Im \frac{L\left[\frac{1+z}{y+\overline{y}z}f\right]}
    {L\left[\frac{i(1-z)}{y+\overline{y}z}f\right]}(0) = -\Re
    \frac{1+\frac{A_{y}}{B_{y}}}{1-\frac{A_{y}}{B_{y}}}<0.
  \end{equation*}
\end{proof}

\bibliographystyle{amsplain}

\bibliography{polyaschurlogconcave}

\end{document}